\documentclass[11pt]{amsart}
\usepackage{mathrsfs}
\usepackage{amssymb}
\usepackage{graphicx}
\usepackage[T1]{fontenc}
\usepackage[all]{xy}
\usepackage{arydshln}
\usepackage{amscd}
\usepackage{amsmath, amssymb}
\usepackage{amsthm}
\usepackage{amsfonts}
\usepackage[colorlinks,linkcolor=blue,citecolor=blue, pdfstartview=FitH]
{hyperref}
\usepackage{backref}
\usepackage{amscd}
\usepackage{easybmat}
\usepackage{mathrsfs}
\usepackage{amsfonts}
\usepackage{color}
\usepackage{pifont}
\usepackage{upgreek}
\usepackage{bm}
\usepackage{hyperref}
\usepackage{shorttoc}
\usepackage{amsmath,amstext,amsthm,a4,amssymb,amscd}
\usepackage[mathscr]{eucal}
\usepackage{mathrsfs}
\usepackage{epsf}
\usepackage{tikz}
\usetikzlibrary{arrows}

\numberwithin{equation}{section}

\def\p{\partial}

\def\mb{\mathbb}

\def\mr{\mathrm}

\def\op{\operatorname}
\def\ra{\rightarrow}
\def\bc{\mathbb C}
\def\br{\mathbb R}
\def\bz{\mathbb Z}
\def\bs{\boldsymbol}
\def\rro{\bs{\rho}}

\def\cal{\mathcal}

\theoremstyle{plain}
\newtheorem{thm}{Theorem}[section]

\newtheorem{prop}[thm]{Proposition}
\newtheorem{cor}[thm]{Corollary}
\theoremstyle{definition}
\newtheorem{rem}[thm]{Remark}

\newtheorem{ntn}[thm]{Notation}

\theoremstyle{definition}

\newtheorem{sthm}{Theorem}
\newcommand{\comment}[1]{}

\usepackage{amscd}
\usepackage{fancyhdr}
\makeatletter
\@namedef{subjclassname@2020}{2020 Mathematics Subject Classification}

\begin{document}

\title{On possible values of the signature of flat symplectic bundles over surfaces with boundary}
\author{InKang Kim}
\author{Pierre Pansu}
\author{Xueyuan Wan}

\address{Inkang Kim: School of Mathematics, KIAS, Heogiro 85, Dongdaemun-gu Seoul, 02455, Republic of Korea}
\email{inkang@kias.re.kr}

\address{Pierre Pansu:
Universit\'e Paris-Saclay, CNRS, Laboratoire de Math\'ematiques d'Orsay\\ 91405 Orsay C\'edex, France}
\email{pierre.pansu@universite-paris-saclay.fr}

\address{Xueyuan Wan: Mathematical Science Research Center, Chongqing University of Technology, Chongqing 400054, China}
\email{xwan@cqut.edu.cn}
\begin{abstract}
We show that every integer in the interval $[2p\chi(\Sigma), -2p\chi(\Sigma)]$ is achieved by the signature of a rank $2p$ flat symplectic bundle over a surface with boundary $\Sigma$. When $p=1$, one can prescribe the type (elliptic, parabolic, hyperbolic) of the holonomy along the boundary.
\end{abstract}

  \thanks{Research by Inkang Kim is partially supported by Grant NRF-2019R1A2C1083865 and KIAS Individual Grant (MG031408). The first author is grateful for the support of IHES during his visits. Pierre Pansu is supported by Agence Nationale de la Recherche, ANR-22-CE40-0004 GOFR, and Xueyuan Wan is partially supported by the National Natural Science Foundation of China (Grant No. 12101093) and the Natural Science Foundation of Chongqing (Grant No. CSTB2022NSCQ-JQX0008), the Scientific Research Foundation of the Chongqing University of Technology.}

\maketitle
\tableofcontents

\section*{Introduction}
The moduli space of flat bundles associated with representations $\phi:\pi_1(\Sigma)\ra G$ has been extensively studied when $\Sigma$ is an orientable surface (with boundary), and $G$ is a semisimple Lie group. Especially its character variety, which is a geometric quotient of the moduli space under the conjugacy action, has drawn many pioneers' attention. Notably N. Hitchin \cite{Hitchin} expressed the variety in algebro-geometric terms, using the notion of Higgs bundle, and more recently F. Labourie \cite{Lab} introduced a dynamical point of view with the notion of Anosov representation. 

When $G$ is the symplectic group, or the isometry group of a nondegenerate Hermitian form, the $1$-cohomology of flat bundles carries a nondegenerate quadratic form, the intersection form, whence a signature. G. Lusztig \cite{Lus} and W. Meyer \cite{Meyer, Meyer1} used the index theorem to express it as an integral. Later, M. Atiyah \cite{Atiyah} extended the discussion to compact surfaces with non-empty boundary. In this context, a result similar to Atiyah-Patodi-Singer's index theorem was obtained in \cite{KPW}. The signature of the flat bundle associated with a representation $\pi_1(\Sigma)\ra G$ where $G$ is a simple Lie group of Hermitian type, decomposes into an integration part, and a boundary term. The integration part is known as the Toledo invariant, and the boundary term is now called a $\rro$-invariant.

The Toledo invariant was first used by D. Toledo \cite{Toledo} to show a rigidity result for closed surface group representations in $U(n,1)$, and later studied by many others to deduce rigidity results when the maximum value is attained \cite{KM1}. M. Burger, A. Iozzi and A. Wienhard \cite{BIW} extended the definition of the Toledo invariant to surfaces with boundary using M. Gromov's bounded cohomology theory.

On the other hand, W. Goldman \cite{Gold1} studied the relative Euler class of flat bundles over surfaces with boundary trivialized canonically along the boundary, and showed that each possible value for the relative Euler number $e$ is attained by some representations whose boundary holonomy is hyperbolic. Indeed, when $\Sigma$ is a closed surface with genus $g$, he showed that the connected components of  $\text{Hom}(\pi_1(\Sigma), PSL(2,\br))$ are $e^{-1}(m)$, where $m$ is an integer satisfying the Milnor-Wood inequality $|m|\leq |\chi(\Sigma)|$. The equality $|e(\phi)|=|\chi(\Sigma)|$ holds if and only if $\phi$ is the holonomy representation of a complete hyperbolic structure on $\Sigma$.

In this paper, we study the possible values of the signature, which is an upgraded version of  the relative Euler number, of flat bundles associated with representations $\pi_1(\Sigma)\ra Sp(2p,\br)$ when $\Sigma$ has boundary. From \cite[Theorem 4 (iii)]{KPW}, we know that it satisfies an inequality of Milnor-Wood type,
\begin{equation}\label{MW}
-2p|\chi(\Sigma)|\le \text{sign}\le 2p|\chi(\Sigma)|.
\end{equation}
For representations to $U(p,p)$, a similar estimate holds.
In this paper, we show that these inequalities are sharp.

\begin{sthm}
\label{main}
Given an orientable surface with boundary $\Sigma$ of negative Euler characteristic and an integer $m$ such that $-2p|\chi(\Sigma)| \leq m \leq 2p|\chi(\Sigma)|$, there exists a representation $\phi:\pi_1(\Sigma)\ra Sp(2p,\br)$ and a representation $\phi':\pi_1(\Sigma)\ra SU(p,p)$ whose signature is $m$.
\end{sthm}

The corresponding problem for unitary groups $U(p,q)$ with $p\not=q$ still eludes us. Here is the answer in the special case when $q=0$.

\begin{sthm}
\label{U(p)}
Let $\Sigma$ be an oriented surface of genus $g$ with $n$ boundary components. Let $p\geq 1$ be an integer. The values achieved by the signature of representations $\pi_1(\Sigma)\to U(p)$ are exactly as follows:
\begin{equation*}
   \begin{cases}
 [-p(n-2),p(n-2)]	& \text{ for }g=0 \text{ and }n\geq 2\\
 [2-np,np-2]\cup \{0\}	&\text{ for } g\geq 1 \text{ and } n\geq 1.
 \end{cases}
\end{equation*}
\end{sthm}

Back to symplectic groups. When $p=1$, more precise results can be given. Conjugacy classes of $SL(2,\br)$ split into elliptic, parabolic and hyperbolic types (defined in Subsection \ref{ellparhyp}). We investigate three subclasses of representations:
\begin{enumerate}
  \item \emph{boundary paraelliptic} ones, i.e. those which have trace in $[-2,2]$ on each boundary component.
  \item \emph{boundary hyperparabolic} ones, i.e. those which are either hyperbolic or parabolic on each boundary component,
  \item \emph{boundary elliptic} ones, i.e. those which are elliptic on each boundary component.
\end{enumerate}
The proof given of Theorem \ref{main} for $p=1$ only uses boundary paraelliptic representations. The other subclasses are dealt with in the following statements.

\begin{sthm}
\label{hyperparabolic}
Let $\Sigma$ be an oriented surface with boundary of negative Euler characteristic. The values achieved by the signatures of boundary hyperparabolic representations $\pi_1(\Sigma)\to SL(2,\br)$ are exactly all integers between $2\chi(\Sigma)$ and $-2\chi(\Sigma)$.
\end{sthm}

\begin{sthm}
\label{elliptic}
Let $\Sigma$ be an oriented surface with nonempty boundary of negative Euler characteristic. The values achieved by the signatures of boundary elliptic representations $\pi_1(\Sigma)\to SL(2,\br)$ are exactly: 
\begin{enumerate}
  \item all even integers between $2\chi(\Sigma)$ and $2|\chi(\Sigma)|$, if $\Sigma$ has genus $>1$, 
  \item all integers of the form $2\chi(\Sigma)+4a$, $a=0,\ldots,|\chi(\Sigma)|$, if $\Sigma$ has genus $0$,
  \item if $\Sigma$ has genus $1$ and $n$ boundary components:
  \begin{itemize}
  \item $-2$ and $2$, if $n=1$,
  \item all even integers between $2\chi(\Sigma)$ and $2|\chi(\Sigma)|$ if $n\ge 2$. 
\end{itemize}
\end{enumerate}
\end{sthm}

The relationship between the maximal signature and the maximal Toledo invariant is subtle.
There are examples where the signature is maximal whereas the Toledo invariant is not the maximum value, $\text{rank}(G) |\chi(\Sigma)|$, see section \ref{relative}. Similarly, a maximal Toledo invariant does not imply that the signature is maximal, see subsection \ref{pair}.

\medskip

{\bf Acknowledgement}  We thank Fanny Kassel for raising the question which led to this paper.

\section{Background for the relative Euler class, the Toledo invariant and the signature}

\subsection{Relative Euler number }
\label{ellparhyp}

Let $\Sigma$ be an oriented surface with boundary with a presentation of the fundamental group $$\pi_1(\Sigma)=\langle A_1,B_1,\cdots,A_g,B_g,C_1,\cdots,C_k: [A_1,B_1]\cdots [A_g,B_g] C_1\cdots C_k=I\rangle$$ where $g$ is the genus of $\Sigma$, and $C_i$ corresponds to a boundary component of $\Sigma$.

Noncentral conjugacy classes in $SL(2,\br)$ are determined by the trace $tr:SL(2,\br)\ra \br$. Assume that $A\in SL(2,\br)$ does not belong to the center $\{\pm I\}$.
\begin{itemize}
\item If $|tr A|>2$, then $A$ is hyperbolic. The set of hyperbolic isometries is denoted by $Hyp$. In fact, $Hyp$ consists of two components, $Hyp_0=\{tr>2\}$ which is contained in a union of one-parameter subgroups, and $Hyp_1=-I Hyp_0=\{tr<-2\}$.
\item If $|tr A|=2$ (and $A\not=\pm I$), then $A$ is parabolic. The set of parabolic isometries is denoted by $Par$. There are two conjugacy classes under $PSL(2,\br)$, depending on whether elements move ideal points in the positive or negative direction on the circle at infinity of the Poincar\'e disc. In $SL(2,\br)$, they lift to 4 conjugacy classes, two of them are contained in $\{tr=2\}$, and they are denoted by $Par+, Par-$. 
\item If $|tr A|<2$, then $A$ is elliptic. The set of elliptic isometries is denoted by $Ell$.
\end{itemize}

\begin{ntn}\label{z}
Let $\tilde G$ be the universal cover of $SL(2,\br)$. Let $Z$ denote the center of $\widetilde G$, this is an infinite cyclic group. Let $\partial$ denote the circle at infinity of hyperbolic plane. Let us fix once and for all an orientation on $\partial$. The elements of $\widetilde G$ act on the universal cover $\tilde\partial$. Let $z\in Z$ be the generator which moves points of $\tilde\partial$ in the positive direction. By abuse of notation, let $Hyp_0$, $Par_0$, $Par_0^+$, $Par_0^-$ be the lifts of $Hyp_0$, $Par\cap \{tr=2\}$, $Par+$, $Par-$ in $SL(2,\br)$ to $\widetilde G$, each of whose closures contains the identity element.
One defines 
$$
Hyp_n=z^n Hyp_0.
$$
Similarly, set
$$
Par_n^\pm=z^n Par_0^\pm.
$$
\end{ntn}
Note that when $|n|$ is odd, the trace of $A\in Hyp_n\cup Par_n$ is negative, where the trace is defined after projection to $SL(2,\br)$.

\medskip

In \cite{Gold1}, W. Goldman calculated the relative Euler class of a homomorphism $\phi:\pi_1(\Sigma)\ra G=PSL(2,\br)$ which is never elliptic on the boundary. This class lies in 
$$
H^2(\Sigma,\partial\Sigma; \bz)=\bz.
$$
Since the boundary holonomy is either hyperbolic or parabolic, there is a canonical trivialization of the corresponding principal $PSL(2,\br)$-bundle along the boundary using one-parameter groups. This trivialization, along a boundary component with holonomy $h$, is given by the canonical section $f:[0,1]\ra G$ such that  $f(t)=h_t^{-1}$ where $h_t$ is a one-parameter group connecting the identity and $h=h_1$. Then it satisfies the equivariance condition
$$
f(0)=id,\ f(1)=h^{-1}id.
$$

The obstruction to extending this boundary trivialization to the whole flat principal $PSL(2,\br)$ bundle lies in $H^2(\Sigma,\partial\Sigma; \bz)=\bz$. The explicit way to calculate this integer is as follows.

Let $\widetilde{\phi(C_i)}$ denote the unique lift of $\phi(C_i)$ which belongs to $\{I\}\cup Hyp_0\cup Par_0\subset \tilde G$. Then for arbitrary lifts $\widetilde{\phi(A_i)},\widetilde{\phi(B_i)}$ in $\tilde G$
\begin{eqnarray}\label{relation}
\prod_{i=1}^g[\widetilde{\phi(A_i)},\widetilde{\phi(B_i)}]\prod_{i=1}^k \widetilde{\phi(C_j)}=z^m
\end{eqnarray}
where, recall Notation \ref{z}, $z$ is the chosen generator of the center of $\tilde G$ lifting $-I$. This integer $m$ is the relative Euler class of $\phi$. Note that the choices
\begin{eqnarray}\label{lift}
\{\widetilde{\phi(A_i)},\widetilde{\phi(B_i)},\widetilde{\phi(C_j)}_{j=1}^{k-1}, z^{-m}\widetilde{\phi(C_k)}\}
\end{eqnarray} determine a lift of $\phi$ to a group homomorphism $\tilde\phi:\pi_1(\Sigma)\ra\tilde G$.

\subsection{Toledo invariant}
Let $\phi:\pi_1(\Sigma)\ra G$ be a homomorphism where $G$ is a Lie group of Hermitian type, i.e., it is connected, semisimple with finite center and no compact factors, and the associated symmetric space is Hermitian with the minimal holomorphic sectional curvature equal to $-1$.

Let $\kappa\in H^2_{c,b}(G, \br)$ be the bounded cohomology class defined by the K\"ahler form $\omega$ by
\begin{equation}\label{bounded cocyle}
  c(g_0,g_1,g_2)=\frac{1}{2\pi}\int_{\triangle(g_0x, g_1x, g_2x)} \omega,
\end{equation}
where $x$ is an arbitrary point and $\triangle(a,b,c)$ is the oriented geodesic triangle in the symmetric space with vertices $a,b,c$.

 By Gromov's isomorphism \cite{Gromov}, one has 
$$
\phi_b^*(\kappa)\in \mathrm{H}^2_b(\pi_1(\Sigma);\mb{R})\cong \mathrm{H}^2_b(\Sigma;\mb{R}).
$$
The canonical map $j_{\p\Sigma}:\mathrm{H}^2_b(\Sigma,\p\Sigma;\mb{R})\to \mathrm{H}^2_b(\Sigma;\mb{R})$ from singular bounded cohomology relative to $\p\Sigma$ to singular bounded cohomology is an isomorphism. Then the Toledo invariant is defined by
$$
\op{T}(\Sigma,\phi)=\langle j^{-1}_{\p\Sigma}\phi_b^*(\kappa),[\Sigma,\p\Sigma]\rangle,$$
where   $j^{-1}_{\p\Sigma}\phi_b^*(\kappa)$ is considered as an ordinary relative cohomology class and $[\Sigma,\p\Sigma]\in \mr{H}_2(\Sigma,\p\Sigma;\mb{Z})\cong\mb{Z}$ denotes the relative fundamental class. 

Note that if $\text{Im}(\phi)$ is contained in an amenable group $H$, for example a Borel subgroup of $G$, then $\phi^*_b(\kappa)=0$ as the following diagram explains,
\begin{equation*}
   \kappa\in H^2_{c,b}(G;\br)\ra H^2_{c,b}(H;\br)=0 \xrightarrow{\phi^*_b} H^2_b(\pi_1(\Sigma);\br).
\end{equation*}
In this case, the Toledo invariant vanishes.

\subsection{Rotation numbers and Toledo invariant}
\label{rot}

Here, we describe an alternate expression of the Toledo invariant, due to M. Burger, A. Iozzi and A. Wienhard, \cite{BIW}. These authors associate a rotation number function with every Lie group $G$ of Hermitian type and every bounded cohomology class $\kappa\in \hat H^2_{cb}(G,\bz)$. The rotation number is a conjugacy invariant continuous function $Rot_\kappa:G\ra \br/\bz$ which extends a certain group homomorphism $u_\kappa:K\ra \br/\bz$, where $K$ is a maximal compact subgroup of $G$. 

The following properties characterize $Rot_\kappa$:
\begin{itemize}
  \item Given an Iwasawa decomposition $G=KAN$, $Rot_\kappa$ vanishes on $AN$. 
  \item If $g\in G$ has Jordan decomposition $g=g_eg_hg_n$, then $Rot_\kappa(g)=Rot_\kappa(g_e)=u_\kappa(k)$ where $g_e$ is conjugate to $k\in K$.
\end{itemize}
Let $\widetilde{Rot_\kappa}:\tilde G\ra\br$ denote the continuous lift of $Rot_\kappa$ to universal covers, such that $\widetilde{Rot_\kappa}(Id)=0$.

\medskip

In \cite[Theorem 12]{BIW},  the following formula is given. Let $\phi:\pi_1(\Sigma)\ra G$ be a representation. Let us use the standard presentation of $\pi_1(\Sigma)$, with $C_1,\ldots,C_k$ representing boundary components. Let $\tilde\phi\in\text{Hom}(\pi_1(\Sigma),\tilde G)$ be a lift of $\phi$ to the universal cover, then
$$
\op{T}_\kappa(\phi)=-\sum_{i=1}^{k} \widetilde{Rot_\kappa}(\tilde\phi(C_j)).
$$

We need this formula only for $G=Sp(2,\br)$, and $\kappa$ the class defined by the K\"ahler form on hyperbolic plane. 
In fact, $\kappa$ is not integral; however, $2\kappa$ is integral, \cite{KPW}, hence one can define the rotation number with $2\kappa$ and then divide everything by 2, this makes no change in the final formula \ref{Rot}. According to \cite[Remark 7.11 (2)]{BIW},  $\kappa$ is associated with the homomorphism $u_\kappa:U(1)\to\br/\bz$ such that $ u_\kappa(k)=k^2$. 
Let us fix a point at infinity $\infty\in\partial$ on the circle at infinity of hyperbolic plane. Then the map $k\mapsto k^2(\infty)$ moves in the positive direction.

Therefore the lift $\widetilde{u_\kappa}:\widetilde{U(1)}\to\br$ takes values $1$ on the generator $z$ of the center of $\widetilde{Sp(2,\br)}$ specified in Notation \ref{z}. In particular, $\widetilde{Rot_\kappa}(z)=1$. Since $\widetilde{Rot_\kappa}$ is equivariant under the deck transformations $Z$, for all $\tilde g\in \widetilde{Sp(2,\br)}$ and $j\in\mathbb{Z}$, 
\begin{equation}\label{Rot}
\widetilde{Rot_\kappa}(z^j\tilde g)=j+\widetilde{Rot_\kappa}(\tilde g).
\end{equation}

\begin{prop}
\label{toleul}
When $n=1$, i.e. for the group $Sp(2,\br)$, the Toledo invariant and the relative Euler number of representations which are nonelliptic on the boundary, coincide.
\end{prop}

Indeed, given a representation $\phi:\pi_{1}(\Sigma)\ra Sp(2,\br)$, let us use the lift $\tilde\phi$ mentioned in Equation \ref{lift}. If $j=1,\cdots, k-1$, the lifts $\tilde\phi(C_i)=\widetilde{\phi(C_i)}$ fit with W. Goldman's calculation, they belong to $Hyp_0\cup Par_0$. Each of them lies on a $1$-parameter subgroup which lifts a hyperbolic or parabolic $1$-parameter subgroup of $Sp(2,\br)$, along which $Rot_\kappa$ is constant and equal to $0$. Therefore
$$
\widetilde{Rot_\kappa}(\tilde\phi(C_j))=\widetilde{Rot_\kappa}(Id)=0.
$$
On the other hand, for $j=k$, $\tilde\phi(C_k)=z^{-m}\widetilde{\phi(C_k)}$, so 
$$
\widetilde{Rot_\kappa}(\tilde\phi(C_k))=\widetilde{Rot_\kappa}(z^{-m}\widetilde{\phi(C_k)})=-m,
$$
and $\op{T}(\phi)=m$. This shows that the Toledo invariant and the relative Euler number coincide.

\subsection{Signature of flat bundles and rho invariant}

Let $\phi:\pi_1(\Sigma)\ra Sp(2n,\br)$ be a representation, and $E=(\br^{2n},\Omega)$ a symplectic vector space. Denote by $\mathcal{E}=\widetilde\Sigma\times_\phi E$ the corresponding flat bundle. The signature of $(\cal E,\Omega)$ is defined as the signature of the symmetric quadratic form
$$
Q(\cdot, \cdot)=\int_\Sigma \Omega ( \cdot \cup \cdot)
$$
on $\text{Im}(H^1(\Sigma, \cal E)\ra H^1(\Sigma,\partial\Sigma, \cal E))$.

Applying Atiyah-Patodi-Singer theory \cite{APS,APSII}, the signature is shown to be (\cite[Theorem 4]{KPW})
\begin{equation}\label{formula}\text{sign}(\phi)=2 \op{T}+\rro_\phi(\p \Sigma),
\end{equation}

The rho invariant $\rro$ is a discontinuous class function on $Sp(2n,\br)$, and can be explicitly calculated \cite{KPW}. One can do it by hand for $Sp(2,\br)=SL(2,\br)$ \cite[Appendix]{KPW}.

We list out some properties of the signature, which are useful for us.
\begin{itemize}
\item If $\phi:\pi_1(\Sigma)\ra Sp(E,\Omega)$ splits into $E=\oplus E_i$, then
$$
\text{sign}(\phi,E)=\sum \text{sign}(\phi,E_i).
$$
\item Let $L\in GL(E)$ be an involution such that $L^*\Omega=-\Omega$. Then $\iota:=Ad_L$ induces an automorphism of $Sp(2n,\br)$ (if $L$ is well chosen, it induces an anti-holomorphic isometry of the symmetric space $Sp(2n,\br)/U(n)$). For all $\phi$,
\begin{equation}\label{anti-holomorphic}
\text{sign}(\iota\circ\phi)=- \text{sign}(\phi).
\end{equation}

\item  If $\Sigma=\Sigma_1\cup_\gamma \Sigma_2$ consists of two subsurfaces glued along a simple closed curve $\gamma$, then
$$\text{sign}(\phi,\Sigma)=\text{sign}(\phi|_{\pi_1(\Sigma_1)},\Sigma_1)+ \text{sign}(\phi|_{\pi_1(\Sigma_2)},\Sigma_2).$$
\end{itemize}
According to \cite[Theorem 4]{KPW}, the signature satisfies the Milnor-Wood inequality
\begin{equation}\label{formula2}
|\text{sign}(\cal E, \Omega)|\leq 2n |\chi(\Sigma)|.
\end{equation}

Here is a table of values of the $\rro$ invariant for the matrix $\begin{pmatrix}
\lambda & \mu\\
0    & \lambda^{-1}
\end{pmatrix}$ in $Sp(2,\br)$. Here $\theta\in (0,\pi)\cup (\pi,2\pi)$.

\begin{center}
\begin{align}\label{etasp}
\begin{tabular}{| c | c | }
\hline
     $\lambda,\mu$   & $\rro_{\phi_0}(S^1)$\\
\hline
   $\lambda\not\in S^1$  & $0$\\
\hline
   $\lambda\in S^1\backslash\{\pm1\}$  &$2(1-\frac{\theta}{\pi})$\\
\hline
    $\mu=0$   & $0$\\
\hline
$\lambda=1, \mu>0$  & $-1$\\
\hline
 $\lambda=1, \mu<0$  & $1$\\
\hline
$\lambda=-1$ & $0$\\
\hline
\end{tabular}	
\end{align}
\end{center}

Thus we see that, when $n=1$, the contribution to the $\rro$ invariant of a boundary component with parabolic holonomy $C$ is $-1$ if $C\in Par+$ and $1$ if $C\in Par-$.

\section{The relative representation variety}\label{relative}
 
 Given a sequence of conjugacy classes $\cal C=(\cal C_1,\cdots,\cal C_k)$ in $Sp(2p,\br)$, the relative representation variety
 $$\text{Hom}^{\cal C}(\pi_1(\Sigma), Sp(2p,\br)):=\{\phi\in \text{Hom}(\pi_1(\Sigma), Sp(2p,\br)):\phi(c_i)\in\cal C_i, 1\leq i \leq k\}$$
is a real semialgebraic set. Note that $\rro$ is constant on $\text{Hom}^{\cal C}(\pi_1(\Sigma), Sp(2p,\br))$, and hence $\text{sign}=2\op{T}+\rro$ is continuous since $\op{T}$ is continuous. On each connected component of $\text{Hom}^{\cal C}(\pi_1(\Sigma), Sp(2p,\br))$, the Toledo invariant is constant, hence the signature is constant as well.

When all $\cal C_i$ are hyperbolic, $\rro=0$, hence $\text{sign}(\phi)=2\op{T}(\phi)$.
Sometimes, the signature behaves differently from the Toledo invariant.
Let us illustrate this for $Sp(2,\br)$. If $\phi$ induces a complete hyperbolic structure with geodesic boundary, then both the Toledo invariant and $\text{sign}(\phi)=2\op{T}(\phi)=2|\chi(\Sigma)|$ are maximal.

The most interesting case is when $\cal C_i$ are classes of elliptic elements.
If $\phi$ induces a hyperbolic surface $S$ with cone singularities with cone angles $0<\phi_i<2\pi$, then 
 \begin{align*}
  \op{T}(\Sigma,\phi)&=\frac{1}{2\pi}\text{Area}(S)=-\left(\chi(S)+\frac{1}{2\pi}\sum_{i=1}^k(\phi_i-2\pi)\right)\\
  &=-\chi(\Sigma)-\sum_{i=1}^k\frac{\phi_i}{2\pi}<|\chi(\Sigma)|.
\end{align*}
But $\text{sign}(\phi)=2\op{T}(\phi)+\rro(\partial\Sigma)=2|\chi(\Sigma)|$ is maximal.
See \cite[Remark 8.7]{KPW} for the calculation of $\rro$. 
The criterion for the existence of a hyperbolic surface with cone singularities is described in \cite{Mc}.

Curiously enough, the signature does not distinguish between a hyperbolic structure with geodesic boundary and a hyperbolic structure with cone singularities whereas the Toledo invariant does.

When $\phi$ induces a hyperbolic surface with parabolic cusps,  the Toledo invariant is maximal $-\chi(\Sigma)$ (\cite[Theorem 3]{BIW}), but the signature is not necessarily maximal. The simplest example is a pair of pants with one parabolic boundary ($\rro=-1$), and two hyperbolic boundaries, which will arise again as a signature $1$ building block in subsection \ref{pair}. 
Hence the signature distinguishes between a hyperbolic surface with geodesic boundaries and a hyperbolic surface with some parabolic boundaries whereas the Toledo invariant does not.

\section{Proof of Theorem \ref{main}}\label{sec-rep}

Since signatures add up in direct sums of representations, the $p=1$ case of Theorem \ref{main} implies the general case using the embedding $SL(2,\br)^{\times p}<Sp(2p,\br)$. The same argument applies to $U(p,p)$, see Remark \ref{upp}. So from now on, we stick to homomorphisms to $G=SL(2,\br)$. Representations to $SO(2)$ turn out to do a lot of mileage. In the next Proposition, item (2) has been included for future reference (Section \ref{boundaryelliptic}), where representations which are trivial on some boundary component are excluded.

\begin{prop}\label{prop-ell}
Let $\Sigma$ be an oriented surface with $n$ boundary components, such that $\chi(\Sigma)\leq 0$. Let $\theta\in(0,\pi)\cup(\pi,2\pi)$ be given.
\begin{enumerate}
  \item The values achieved by the signatures of $SO(2)$-representations with prescribed conjugacy class $\theta$ at some boundary component are exactly all even integers between $4-2n$ and $2n-4$. In fact, one can prescribe that at most one boundary holonomy be trivial, all others being elliptic.
  \item The values achieved by the signatures of boundary elliptic $SO(2)$-represen\-tations with prescribed conjugacy class $\theta$ at some boundary component are exactly all integers of the form $2n-4a$, $a=1,\ldots,n-1$.
\end{enumerate}
\end{prop}

\begin{proof}
Representations $\phi:\pi_1(\Sigma)\to SO(2)$ have vanishing Toledo invariant, by \cite[Remark 8.3]{KPW}, hence
\begin{equation*}
  \mr{sign}(\phi)=\rro_\phi(\p\Sigma).
\end{equation*}

Let $A_i$, $i=1,\ldots,g$, $B_j$, $j=1,\ldots,g$, and $C_k$, $k=1,\ldots,n$ be the standard generators of $\pi_1(\Sigma)$, subject to the single relator $\prod_{i=1}^{g}[A_i,B_i]\prod_{k=1}^{n}C_k=I$. Since $SO(2)$ is abelian, representations are in $1-1$ correspondance with $2g+n$-tuples of angles $\theta_i,\theta_j,\theta_k\in[0,2\pi)$, subject to the single relation $\sum_{k=1}^{n}\theta_k\equiv 0$ mod $2\pi$. We note that $\theta_i,\theta_j$ will play no role at all. 

Given a representation $\phi:\pi_1(\Sigma)\to SO(2)$, let $a$ be the integer such that $\sum_{k=1}^{n}\theta_k =2\pi a$. If the representation is boundary elliptic, then $\theta_k\notin\{0,\pi\}$ for all $k$, and
$$
\rro_\phi(\p\Sigma)=\sum_{k=1}^{n}2(1-\frac{\theta_k}{\pi})=2n-4a,\quad 0<a<n.
$$
Without the ellipticity requirement, $\theta_k$ can take the values $0$ and $\pi$. The occurrence of $\pi$ does not affect the expression of $\rro_\phi(\p\Sigma)$. Thus $\rro_\phi(\p\Sigma)=2m-4a$ where $m$ is the number of nonzero $\theta_k$'s, and $0<a<m$. This gives more even numbers, but still between $4-2n$ and $2n-4$.

Conversely, given $\theta\in(0,\pi)\cup(\pi,2\pi)$, we first choose $\theta_1=\theta$ and all other $\theta_k=\frac{2\pi a-\theta}{n-1}$. This allows to achieve every $2n-4a$ for $0<a<n$ as the signature of a boundary elliptic $SO(2)$-representation with a prescribed elliptic holonomy on one boundary component.

Secondly, we choose $\theta_1=\theta$,  and for any  $0\leq |m| \leq n-2$,\begin{equation}\label{eqn1}
  \theta_1+\cdots+\theta_{|m|+2}= \begin{cases}
 	2\pi& \text{ if } m\geq 0 \\
 	2(1+|m|)\pi&\text{ if } m\leq 0
 \end{cases},\quad \theta_i\neq 0,\pi\text{ for }1\leq i\leq |m|+2,
\end{equation}
and $\{\theta_{|m|+3},\cdots,\theta_{n}\}$ satisfy 
\begin{equation}\label{eqn2}
\theta_{|m|+3},\cdots,\theta_{n-1}\neq 0,\pi\text{ and }  \theta_{|m|+2i+1}+\theta_{|m|+2i+2}=2\pi,\,i\geq 1,
\end{equation}
where  $\theta_{n+1}:=2\pi$. By the choice of $\theta_i$, we know that
$\phi(C_1),\cdots,\phi(C_{n-1})$ are elliptic, and $\phi(C_n)$ is elliptic or $I$. In particular, $\phi(C_n)=I$ if and only if $n-|m|$ is odd. Moreover,
\begin{equation}\label{eqn3}
  \sum_{i=1}^n\mathrm{sgn}(\theta_i)-\frac{1}{\pi}\sum_{i=1}^n\theta_i=|m|+2-\frac{1}{\pi}( \theta_1+\cdots+\theta_{|m|+2})=m.
\end{equation}
Hence  $$\rro_\phi(\p\Sigma)=\sum_{i=1}^n2(\mathrm{sgn}(\theta_i)-\frac{\theta_i}{\pi})=2\left(\sum_{i=1}^n\mathrm{sgn}(\theta_i)-\frac{1}{\pi}\sum_{i=1}^n\theta_i\right)=2m.$$
\end{proof}

\subsection{Achieving odd signatures in genus $0$}

Let $\Sigma$ be an oriented surface of genus zero, with $n\geq 3$ boundary components. Proposition \ref{prop-ell} shows that all even numbers in $[-2|\chi(\Sigma)|,2|\chi(\Sigma)|]$ arise as signatures of $SO(2)$ representations on $\Sigma$. One can even prescribe the conjugacy class of the holonomy of one boundary component. We now show that all odd numbers in this interval can also be achieved.

First, if $n=3$, there exist $SL(2,\br)$-representations $\phi_1$ and $\phi_2$ with two elliptic and one parabolic boundary holonomies and signatures $-1$ and $1$ respectively. Indeed, we consider the following representations:
\begin{equation}\label{eqn5}
  \phi_-(d_1)=\begin{pmatrix}
  0& -1\\
  1&0 
\end{pmatrix},\quad \phi_-(d_2)=\begin{pmatrix}
  1&1 \\
  0& 1
\end{pmatrix}.
\end{equation}
Then 
\begin{equation}\label{eqn6}
  \phi_-(d_1)\phi_-(d_2)=\begin{pmatrix}
  0&-1 \\
  1&1 
\end{pmatrix}.
\end{equation}
From \cite[(5.4) and Section 6]{Atiyah}, one has
\begin{equation*}
  \mathrm{sign}(\phi_-)=-\frac{4}{3}-(-1)-\frac{2}{3}=-1. 
\end{equation*}
Combining with \cite[(5.6)]{Atiyah}, one has
$$\mathrm{sign}(\phi_+)=1\text{ for }\phi_+(d_1)=\begin{pmatrix}
  0&1\\
  -1&0 
\end{pmatrix}\text{ and }\phi_+(d_2)=\begin{pmatrix}
  1&-1 \\
  0& 1
\end{pmatrix}.
$$

\begin{prop}\label{prop-ell1}
If $\Sigma$ has genus zero, then every integer in $[-2|\chi(\Sigma)|,2|\chi(\Sigma)|]$ is the signature of some representation $\pi_1(\Sigma)\to SL(2,\mb{R})$ whose boundary holonomies are elliptic except at most one of them which could be parabolic or trivial.
\end{prop}

\begin{proof}
Proposition \ref{prop-ell} shows that all even numbers in $[-2|\chi(\Sigma)|,2|\chi(\Sigma)|]$ arise as signatures of $SO(2)$ representations of $\pi_1(\Sigma)$ which are elliptic except at most one which could be trivial. 

Let $m$ be an odd integer between $0$ and $2|\chi(\Sigma)|=2n-4$. If $m\equiv 2n+1$ mod $4$, there exists an integer $a$ such that $m=2n-4a+1$. Then $2\le a\le \frac{2n+1}{4}<n$. If $m\equiv 2n-1$ mod $4$, there exists an integer $a$ such that $m=2n-4a-1$. Then $1\le a\le \frac{2n-1}{4}<n$ again. 

We view $\Sigma=\Sigma'\#\Sigma''$ as the gluing of a genus $0$ surface $\Sigma'$ with $n-1$ boundary components, and a pair of pants $\Sigma''$, along a common boundary component $c$. According to Proposition \ref{prop-ell}, $\Sigma'$ admits a representation $\phi'$ with signature $2n-4a$, with holonomy along $c$ which can be prescribed to be conjugate either to $\phi_+(d_1)^{-1}$ or to $\phi_-(d_1)^{-1}$. The representation of $\pi_1(\Sigma)$ obtained by gluing $\phi'$ and $\phi_+$ (resp. $\phi_-$) has signature $2n-4a+1=m$ (resp. $2n-4a-1=m$), and all boundary components have elliptic holonomies but one which is parabolic. 

Conjugating with an involution of $SL(2,\br)$ allows to achieve all negative integers $\ge 2\chi(\Sigma)$.
\end{proof}

\subsection{Achieving odd signatures, higher genus}
 
Let $\Sigma_{1,1}$ be an oriented surface of genus one with one boundary component. For every $\theta\in(0,\pi)$, there exists a hyperbolic structure on $\Sigma_{1,1}$ with a cone singularity of angle $2\pi-2\theta$. Its holonomy representation $\psi_\theta:\pi_1(\Sigma_{1,1})\to PSL(2,\mb{R})$ admits a lift to $SL(2,\mb{R})$ whose value on the (oriented) boundary component is the matrix $\begin{pmatrix}
\cos\theta & -\sin\theta   \\
\sin\theta  &  \cos\theta 
\end{pmatrix}$ and whose signature is equal to $2$, see Section \ref{relative}. 

Let $\Sigma$ be a surface of genus $g$ with $n$ boundary components, where $n\geq 1$ and $g\geq 1$. Then $\Sigma$ can be decomposed as the connected sum of $g$ copies of $\Sigma_{1,1}$ and a surface $\Sigma_{n+g}$ of genus zero with $n+g$ boundary components,
\begin{equation*}
  \Sigma=\Sigma_{n+g}\#\underbrace{\Sigma_{1,1}\#\cdots\#\Sigma_{1,1}}_{g}.
\end{equation*}
The case when $g=5$ and $n=4$ is depicted on Figure 1.
\begin{figure}[ht]
\centering
\includegraphics[width=0.68\textwidth]{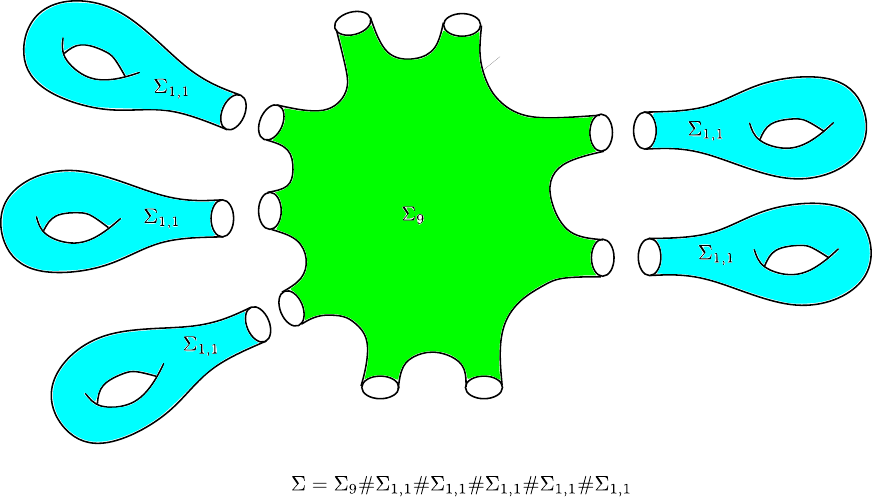}
\caption{$\Sigma$ is  the connected sum of $\Sigma_9$ and five $\Sigma_{1,1}$.}
\end{figure}

Let us fix an integer $m\geq 0$ in $[-2(n-2),2(n+2g-2)=2|\chi(\Sigma)|]$. Then
 \begin{equation*}
  -2(n+g-2)\leq m-2g\leq 2(n+g-2)=2|\chi(\Sigma_{n+g})|.
\end{equation*}
By Proposition \ref{prop-ell1}, there exists a representation 
$\phi_{g+1}:\pi_1(\Sigma_{n+g})\to SL(2,\mb{R})$ such that 
\begin{equation*}
  \mathrm{sign}(\phi_{g+1})=m-2g,
\end{equation*}
and such that at least $n+g-1\geq g$ boundary components $c_1,\cdots,c_{n+g-1}$ have elliptic holonomies. Suitably adjusting the angles $\theta_j$, one can glue representations
\begin{equation*}
  \psi_{\theta_j}:\pi_1(\Sigma_{1,1})\to SL(2,\mb{R}),\quad 1\leq j\leq g,
\end{equation*}
and obtain a representation of $\pi_1(\Sigma)$ of signature $m-2g+g\times 2=m$.

By composing with an involution $\iota$ of $SL(2,\br)$ (Formula \ref{anti-holomorphic}), one achieves negative integers $\ge 2\chi(\Sigma)$.

Hence we obtain
\begin{thm}\label{app-thm1}
Let $\Sigma$ be a surface with $n\geq 1$ boundaries, and $\chi(\Sigma)\leq 0$.
Then every integer in $[-2|\chi(\Sigma)|,2|\chi(\Sigma)|]$ is the signature of some boundary paraelliptic representation $\pi_1(\Sigma)\to SL(2,\mb{R})$, i.e. whose holonomies on boundary components have traces in $[-2,2]$. 
\end{thm}

\begin{rem}\label{upp}
Let $\Sigma$ be an oriented surface with boundary and negative Euler characteristic. Let $p\ge 1$ be an integer. The signatures of representations $\pi_1(\Sigma)\to SU(p,p)$ or $U(p,p)$ are all integers between $2p\chi(\Sigma)$ and $-2p\chi(\Sigma)$.
\end{rem}
Indeed, the Milnor-Wood type inequality of \cite[Theorem 4]{KPW} for $U(p,p)$-representations takes the form
\begin{equation*}
  |\mathrm{sign}|\leq 2p|\chi(\Sigma)|.
\end{equation*}
Since $SU(1,1)\cong SL(2,\mb{R})$, $SL(2,\mb{R})^p<SU(p,p)$, so, by additivity of signature under direct sums, every integer in the interval $[-2p|\chi(\Sigma)|,2p|\chi(\Sigma)|]$ can be the signature of some representation in $SU(p,p)$, hence in $U(p,p)$. This completes the proof of Theorem \ref{main}.

\section{Scheme of the proof of Theorem \ref{hyperparabolic}}

The higher genus case will be handled first, in Proposition \ref{higher}. We rely on W. Goldman's theorem that all integers in $[-|\chi(S)|,|\chi(S)|]$ are Euler classes of symplectic bundles over a closed surface $S$. When $\Sigma$ is the complement of a collection of disks in $S$, the required representations are obtained by deforming representations $\phi$ of $\pi_1(S)$. The argument requires $\phi$ to be a smooth point of $\text{Hom}(\pi_1(S),G)$. Unfortunately, we can achieve this only when the genus is at least $2$.

The genus 0 and 1 cases are proved in Proposition \ref{zero}. This time, the additivity of signature under surgery is used. This reduces to build flat bundles with prescribed signature over the pair of pants and the punctured torus.



\section{Higher genus}
\label{highergenus}

\subsection{One boundary component case}

Let $G=SL(2,\br)$.
Let $\Sigma$ be a genus $g\geq 2$ once punctured surface with boundary $C$. Let $S$ be a closed surface such that $\Sigma$ is homeomorphic to the complement of a disc in $S$. Then $\text{Hom}(\pi_1(S), G)$ can be identified with $\Phi^{-1}(id)$ where
$$
\Phi:G^{2g}\ra G,\ \Phi(A_1,B_1,\cdots,A_g,B_g)=\prod_{i=1}^{g} [A_i, B_i].
$$
Any element $\phi\in \text{Hom}(\pi_1(S), G)$ can be viewed as an element in $\text{Hom}(\pi_1(\Sigma), G)$ with $\phi(C)=I$. If $\phi \in \text{Hom}(\pi_1(S), G)$ is a regular point of $\Phi$, then a neighborhood $N$ of $\phi$ surjects onto a neighborhood $\cal N$ of $I$ in $G$. Choose a path $C_t\subset Par$ or $Hyp$ in $G$ with $C_0=id$. Then there exists a path $\phi_t$ with $\phi_0=\phi$ such that 
$$
\prod_{i=1}^{g} [\phi_t(A_i), \phi_t(B_i)]=C_t.
$$
This gives rise to a representation
$$
\phi_t:\pi_1(\Sigma)\ra G
$$
such that $\phi_t(C)=C_t^{-1}$.

Since the Toledo invariant is continuous and coincides with the relative Euler number of $\phi_t$ when the boundary holonomy is not an elliptic element,
$$\op{T}(\phi_t)=\op{T}(\phi)$$ for all $t$.
Using the formula \ref{formula}, $\text{sign}(\phi_t)=2 \op{T}(\phi_t)+\rro(C_t^{-1})$, we get, for $t>0$,
$$
\text{sign}(\phi_t)=2 \op{T}(\phi) +\epsilon, 
$$
with $\epsilon=-1$ if $C_t^{-1}\in Par+$, $\epsilon=0$ if $C_t^{-1}\in Hyp$ and $\epsilon=1$ if $C_t^{-1}\in Par-$. In this way, we get representations $\phi:\pi_1(\Sigma)\ra SL(2,\br)$ with prescribed signatures.  

In his thesis \cite{Gold1}, W. Goldman showed that when genus is at least $2$, every integer between $\chi(S)$ and $-\chi(S)$ can be realized as the Euler number, hence the Toledo number of some representation $\phi\in \text{Hom}(\pi_1(S),G)$. Note that $\chi(\Sigma)=\chi(S)-1$. Hence one can realize every integer between $2\chi(S)-1=-2|\chi(\Sigma)|+1$ and $-2\chi(S)+1=2|\chi(\Sigma)|-1$ as the signature of some representation $\phi:\pi_1(\Sigma)\ra SL(2,\br)$ by the above method. 

The maximum value $2|\chi(\Sigma)|$ can be obtained as the signature of the holonomy of a hyperbolic structure on $\Sigma$ with geodesic boundary, where the Toledo invariant is $\text{Area}(\Sigma)/2\pi=|\chi(\Sigma)|$.

Hence if we can find a regular point $\phi$ realizing a prescribed even signature, we shall be done.

\medskip

Let $G$ be a semisimple real algebraic group. In \cite{KP}, a point $\phi$ in $\text{Hom}(\pi_1(S), G)$ is called \emph{smooth} if $$\dim(Z^1(\pi_1(S), \mathfrak{g}_{Ad\circ\phi})=(1-\chi(S))\dim(G)=(2g-1)\dim G.$$
According to \cite[Proposition C.2]{KP}, this is equivalent to 
$$\Phi: G^{2g}\ra G$$ being a submersion at $\phi$.

By \cite[Lemma 2.3]{KP}, if $\phi$ has a discrete centralizer, then $H^2(\pi_1(S), \mathfrak{g}_{Ad\circ\phi})=0$, and $\phi$ is a smooth point. 

Now back to the case $G=SL(2,\br)$. If $\phi$ has a non-discrete centralizer, the image $\phi(\pi_1(S))\subset G$ of $\phi$ must be contained in an one-parameter subgroup of $G$. This implies that the Toledo invariant vanishes. Hence every $\phi\in \text{Hom}(\pi_1(S),G)$ with non-zero Toledo invariant is a smooth point, and so a regular point of $\Phi^{-1}(id)$.

If genus is $>1$, one can construct representations $\phi:\pi_1(S)\ra G$ with zero Toledo invariant, but with a discrete centralizer. Let $\phi(A_1)=h_1, \phi(A_2)=h_2$, and all other generators are mapped to $I$, where $h_1,h_2$ are two hyperbolic isometries of hyperbolic plane sharing exactly one fixed point at infinity. It satisfies the relation $\prod [\phi(A_i),\phi(B_i)]=I$. Hence $\phi$ is a homomorphism whose image lies in a Borel subgroup. Its Toledo invariant is zero, and the centralizer is trivial. Hence $\phi$ is a smooth point.

In this way, one can realize every integer between $2\chi(\Sigma)$ and $-2\chi(\Sigma)$ as the signature of a representation in $Sp(2,\br)$.

\subsection{Multiple boundary component case}
For a surface $\Sigma=S\setminus \{D_1,\cdots,D_k\}$ with many boundary components $C_1,\cdots,C_k$, one uses the map
$$
\Phi:G^{2g+k-1}\ra G^k,
$$
$$
\Phi(A_i,B_i, C_1,\cdots,C_{k-1})=(\prod[A_i,B_i]C_1\cdots C_{k-1},C_1,\cdots, C_{k-1}).
$$
Note that $\pi_1(\Sigma)=\langle A_i, B_i, C_1,\cdots, C_k: \prod_{i=1}^g[A_i,B_i]C_1\cdots C_{k-1}=C_k^{-1} \rangle$. Hence any homomorphism $\phi:\pi_1(S)\ra G$ is in $\Phi^{-1}(id,\cdots,id)$.  Since $\Phi$ on the last $k-1$ coordinates is merely a projection, it is obviously a submersion. On the first $2g$ coordinates, one uses the previous argument. 

Viewing $\phi$ as a representation $\phi:\pi_1(\Sigma)\ra G$ with $\phi(C_i)=id$, one can deform it to a nearby one whose holonomy around each puncture becomes parabolic or hyperbolic. 
Since $|\chi(\Sigma)|=|\chi(S)|+k$, we can fill in the gap by adding or subtracting by $1=\pm\rro(\text{parabolic})$ up to $2|\chi(S)|+k=2|\chi(\Sigma)|-k$ using Goldman's result.

Take a hyperbolic structure on $\Sigma$ with geodesic boundary. Then the Toledo invariant is $|\chi(\Sigma)|$. Let $\phi$ be the holonomy representation of this hyperbolic structure satisfying $\prod [\phi(A_i),\phi(B_i)]\phi(C_1)\cdots \phi(C_{k-1})=\phi(C_k)^{-1}$. Then deform one by one each of the $\phi_t(C_j),j=1,\cdots,k-1$ to a parabolic element $P_j$ with $\rro(P_j)=-1$, while staying in $Hyp_0$.
Geometrically it can be done by pinching the boundary component $C_j$ to a parabolic one. Since the Toledo invariant is a continuous function, and the relative Euler class has integer values when the boundary holonomy is non-elliptic, the Toledo invariant remains equal to $|\chi(\Sigma)|$ throughout the deformation. When $\phi$ becomes $\phi_1$ which hits the parabolic element $P_1$,
the signature changes to
$$
\text{sign}(\phi_1)=2 \op{T}(\phi_1)+\sum\rro(\phi_1(C_i))=2|\chi(\Sigma)|-1.
$$
We repeat this process one by one by pinching each boundary component $C_j, j=1,
\cdots, k-1$ to a parabolic one, to achieve all the values $2|\chi(\Sigma)|, 2|\chi(\Sigma)|-1,\cdots, 2|\chi(\Sigma)|-(k-1)$.

In this way, provided the genus is at least $2$, we obtain all the values from 0 to $2|\chi(\Sigma)|$, and by Formula \ref{anti-holomorphic}, we get all the negative values.

Therefore we have proved

\begin{prop}
\label{higher}
Let $\Sigma$ be a compact orientable surface with boundary with genus at least $2$. Then every integer in the interval $[-2|\chi(\Sigma)|,2|\chi(\Sigma)|]$ is the signature of some representation $\pi_1(\Sigma)\to Sp(2,\br)$ which is hyperbolic or parabolic on the boundary.
\end{prop}

\section{Building blocks}

To achieve the possible values in $Sp(2,\br)$ in genus 0 and 1, we assemble representations from Euler characteristic $-1$ subsurfaces. By additivity of signature, it therefore suffices to achieve signatures $-2,-1,0,1,2$ (plus some requirement on the boundary) on Euler characteristic $-1$ subsurfaces. Since composing a representation with an involution of $Sp(2,\br)$, like in Formula \ref{anti-holomorphic}, changes the sign of signature, we shall focus on the values $0,1,2$. All through, we shall need to compute Toledo invariants. For this, we shall use Proposition \ref{toleul} and compute relative Euler classes instead. Then we shall compute $\rro$ invariants, and obtain the signature from Formula \ref{formula}.

There are two surfaces of Euler characteristic $-1$, the pair of pants and the punctured torus, whence six building blocks.

\subsection{Pair of pants}\label{pair} A pair of pants $P$ has three boundary components $A,B,C$ and the induced orientation on the boundary is such that    $ABC=id$ at the fundamental group level. Here we use $A,B,C$ to denote both the boundary components, the corresponding elements of the fundamental group and their images in $Sp(2,\br)$. To find a representation $\phi:\pi_1(P)\ra Sp(2,\br)$, it suffices
to give three elements in $Sp(2,\br)$ satisfying $\phi(A)\phi(B)\phi(C)=I$. We shall use Formula \ref{formula},
\begin{equation*}\label{formulab}
\text{sign}(\phi)=2 \op{T}+\rro_\phi(\p P),
\end{equation*}
and the bound provided by the Milnor-Wood inequality \ref{MW}
 $$
 |\text{sign}(\phi)|\leq 2.
 $$\\

$\bullet$ Signature 0 building blocks\\

Consider two hyperbolic isometries fixing exactly one common ideal point, say $\infty$.
Concretely one can take
$$A=\begin{pmatrix}
\lambda & 0\\
0 & \lambda^{-1}\end{pmatrix}, B=\begin{pmatrix}
                                                             \alpha &  c \\
                                                              0         & \alpha^{-1}\end{pmatrix}.$$
Then $AB=\begin{pmatrix}
                 \lambda \alpha & \lambda c\\
                 0                         & (\lambda\alpha)^{-1}\end{pmatrix}$.                                                              
One can arrange so that $\lambda\alpha>1$ or $\lambda\alpha<-1$ by choosing $\alpha$ properly. Since the group generated by $A$ and $B$ is contained in the Borel subgroup, the Toledo invariant of this representation is zero.  By taking $C=(AB)^{-1}$, the signature is zero also for these representations by (\ref{formula}). We collect this information in the following diagram.
 \begin{center}
\begin{tikzpicture}[x=1cm,y=1cm]



\draw[thick] (0,0.3) -- (1,0.8);
\draw[thick] (0,-0.3)-- (1, -0.8);
\draw[thick] (1, 0.25)--(1, -0.25);

\draw [xscale=cos(70),  thick] (0,-0.3) arc (-90:90:0.3);
\draw [xscale=cos(70), thick] (0,0.3) arc (90:270:0.3);

\draw [xscale=cos(70),  thick] (3,0.2) arc (-90:90:0.3);
\draw [xscale=cos(70), thick] (3,0.8) arc (90:270:0.3);

\draw [xscale=cos(70),  thick] (3,-0.8) arc (-90:90:0.3);
\draw [xscale=cos(70), thick] (3,-0.2) arc (90:270:0.3);



\draw    (0, 0.6) node{{\tiny $t>0$}};
\draw    (1, 1) node{{\tiny $t>0$}};
\draw    (1, -1) node {{\tiny $t>0$}};
\draw    (0.5, 0) node{{\tiny 0}};

\draw[thick] (0,0.3) -- (1,0.8);
\draw[thick] (0,-0.3)-- (1, -0.8);
\draw[thick] (1, 0.25)--(1, -0.25);

\draw [xscale=cos(70),  thick] (9,-0.3) arc (-90:90:0.3);
\draw [xscale=cos(70), thick] (9,0.3) arc (90:270:0.3);

\draw [xscale=cos(70),  thick] (12,0.2) arc (-90:90:0.3);
\draw [xscale=cos(70), thick] (12,0.8) arc (90:270:0.3);

\draw [xscale=cos(70),  thick] (12,-0.8) arc (-90:90:0.3);
\draw [xscale=cos(70), thick] (12,-0.2) arc (90:270:0.3);
\draw    (3, 0.6) node{{\tiny $t>0$}};
\draw    (4, 1) node{{\tiny $t<0$}};
\draw    (4, -1) node {{\tiny $t<0$}};
\draw    (3.5, 0) node{{\tiny 0}};
\draw[thick] (3.1,0.3) -- (4.1,0.8);
\draw[thick] (3.1,-0.3)-- (4.1, -0.8);
\draw[thick] (4.1, 0.25)--(4.1, -0.25);
\end{tikzpicture}
\end{center}        
                Here the number 0 inside the pair of pants denotes the signature,  $t>0$ represents a hyperbolic element with positive trace $>2$, and $t<0$ represents a hyperbolic element with negative trace $<-2$. 
                
                By taking $\lambda\alpha=-1$, we get $C$ a parabolic element with negative trace so that $\rro(C)=0$.
                   \begin{center}
\begin{tikzpicture}[x=1cm,y=1cm]

\draw[thick] (0,0.3) -- (1,0.8);
\draw[thick] (0,-0.3)-- (1, -0.8);
\draw[thick] (1, 0.25)--(1, -0.25);

\draw [xscale=cos(70),  thick] (0,-0.3) arc (-90:90:0.3);
\draw [xscale=cos(70), thick] (0,0.3) arc (90:270:0.3);


\draw [xscale=cos(70),  thick] (3,-0.8) arc (-90:90:0.3);
\draw [xscale=cos(70), thick] (3,-0.2) arc (90:270:0.3);

\draw    (0, 0.6) node{{\tiny $t>0$}};
\draw    (1, 1) node{{\tiny $\rro=0$}};
\draw    (1, -1) node {{\tiny $t<0$}};
\draw    (0.5, 0) node{{\tiny 0}};

\draw[thick] (0,0.3) -- (1,0.8);
\draw[thick] (0,-0.3)-- (1, -0.8);
\draw[thick] (1, 0.25)--(1, 0.8);

\end{tikzpicture}
\end{center}

$\bullet$ Signature 1 building blocks\\

 Take
$$A=\begin{pmatrix}
\lambda & 0\\
0 & \lambda^{-1}\end{pmatrix}, B=\begin{pmatrix}
                                                             \lambda^{-1} &  c \\
                                                              0         & \lambda\end{pmatrix}, \lambda>1, c>0$$ so that 
                                                               $AB=\begin{pmatrix}
                 1 & \lambda c\\
                 0                         & 1 \end{pmatrix}$ with $\lambda c>0$.                                                              
For the same reason, the Toledo invariant is 0. Hence $(AB)^{-1}=C=\begin{pmatrix}
 1 & -\lambda c \\
    0  &  1 \end{pmatrix}$, and hence $\rro(C)=1$
according to the table \ref{etasp}.
        \begin{center}
\begin{tikzpicture}[x=1cm,y=1cm]

\draw[thick] (0,0.3) -- (1,0.8);
\draw[thick] (0,-0.3)-- (1, -0.8);
\draw[thick] (1, 0.25)--(1, -0.25);

\draw [xscale=cos(70),  thick] (0,-0.3) arc (-90:90:0.3);
\draw [xscale=cos(70), thick] (0,0.3) arc (90:270:0.3);


\draw [xscale=cos(70),  thick] (3,-0.8) arc (-90:90:0.3);
\draw [xscale=cos(70), thick] (3,-0.2) arc (90:270:0.3);

\draw    (0, 0.6) node{{\tiny $t>0$}};
\draw    (1, 1) node{{\tiny $\rro=1$}};
\draw    (1, -1) node {{\tiny $t>0$}};
\draw    (0.5, 0) node{{\tiny 1}};

\draw[thick] (0,0.3) -- (1,0.8);
\draw[thick] (0,-0.3)-- (1, -0.8);
\draw[thick] (1, 0.25)--(1, 0.8);

\end{tikzpicture}
\end{center}                  
        Here $\rro=1$ represents a parabolic element with $\rro=1$.       
        
Consider a parabolic element fixing $\infty$, $A=\begin{pmatrix}
                                               1 & n \\
                                               0 &  1\end{pmatrix}$ with $n>0$ so that $\rro(A)=-1$ and a hyperbolic element fixing $0$, $B=\begin{pmatrix}
                                               a & 0 \\
                                               c  & a^{-1}\end{pmatrix}$ with $a>1$ so that
   $AB=\begin{pmatrix}
   a+nc & na^{-1}\\
   c      & a^{-1}\end{pmatrix}$.         Then choosing $c$ properly, the trace of $AB$,
   $a+nc+a^{-1}<-2 $, or equal to $-2$.                             The former case, $AB$ is hyperbolic with negative trace, and the latter case is parabolic.    The Toledo invariant is 1 in both cases.     Take $C= (AB)^{-1}$. Then either $\text{tr}(C)<0$ or $\rro(C)=0$.
        \begin{center}
\begin{tikzpicture}[x=1cm,y=1cm]


\draw [xscale=cos(70),  thick] (3,0.2) arc (-90:90:0.3);
\draw [xscale=cos(70), thick] (3,0.8) arc (90:270:0.3);

\draw [xscale=cos(70),  thick] (3,-0.8) arc (-90:90:0.3);
\draw [xscale=cos(70), thick] (3,-0.2) arc (90:270:0.3);

\draw    (-0.5, 0.2) node{{\tiny $\rro=-1$}};
\draw    (1, 1) node{{\tiny $t<0$}};
\draw    (1, -1) node {{\tiny $t>0$}};
\draw    (0.5, 0) node{{\tiny 1}};

\draw[thick] (0,0) -- (1,0.8);
\draw[thick] (0,0)-- (1, -0.8);
\draw[thick] (1, 0.25)--(1, -0.25);


\draw [xscale=cos(70),  thick] (12,-0.8) arc (-90:90:0.3);
\draw [xscale=cos(70), thick] (12,-0.2) arc (90:270:0.3);

\draw    (2.7, 0.2) node{{\tiny $\rro=-1$}};
\draw    (4, 1) node{{\tiny $\rro=0$}};
\draw    (4, -1) node {{\tiny $t>0$}};
\draw    (3.7, 0) node{{\tiny 1}};

\draw[thick] (3.1,0) -- (4.1,0.8);
\draw[thick] (3.1,0)-- (4.1, -0.8);
\draw[thick] (4.1, 0.8)--(4.1, -0.25);

\end{tikzpicture}
\end{center}                         
                 
$\bullet$ Signature 2 and more signature 0 building blocks\\

Take two hyperbolic elements $A,B$ with positive traces, which generate a Schottky group, whose quotient is a pair of pants. In this case, the representation is the holonomy representation of a complete hyperbolic structure, and the Toledo invariant is 1. By (\ref{relation}), if we take $C'$ with positive trace, then $ABC'=-I$ since the Toledo invariant $m=1$. Hence $C=-C'$. Hence $C=(AB)^{-1}$ is a hyperbolic isometry with negative trace. 

 \begin{center}
\begin{tikzpicture}[x=1cm,y=1cm]



\draw[thick] (0,0.3) -- (1,0.8);
\draw[thick] (0,-0.3)-- (1, -0.8);
\draw[thick] (1, 0.25)--(1, -0.25);

\draw [xscale=cos(70),  thick] (0,-0.3) arc (-90:90:0.3);
\draw [xscale=cos(70), thick] (0,0.3) arc (90:270:0.3);

\draw [xscale=cos(70),  thick] (3,0.2) arc (-90:90:0.3);
\draw [xscale=cos(70), thick] (3,0.8) arc (90:270:0.3);

\draw [xscale=cos(70),  thick] (3,-0.8) arc (-90:90:0.3);
\draw [xscale=cos(70), thick] (3,-0.2) arc (90:270:0.3);



\draw    (0, 0.6) node{{\tiny $t>0$}};
\draw    (1, 1) node{{\tiny $t<0$}};
\draw    (1, -1) node {{\tiny $t>0$}};
\draw    (0.5, 0) node{{\tiny 2}};
\end{tikzpicture}
\end{center}   
We can deform $A,B$ to parabolic elements to get two cusped pair of pants.
Concretely, one can take one parabolic element fixing $\infty$, and the other parabolic element fixing $0$, to obtain
$$\begin{pmatrix}
1 & n \\
0 & 1 \end{pmatrix} \begin{pmatrix}
                                  1 & 0 \\
                                   m & 1 \end{pmatrix}=\begin{pmatrix}
                                                               1+nm & n \\
                                                               m    & 1 \end{pmatrix}.$$
Here $n>0$ so that $\rro(A)=-1$, $m<0$  so that the trace of $AB$, $2+mn<-2$.
Note that $\rro(B)=-1$ since  one can conjugate it by $-1/z$ to
$\begin{pmatrix}
1 & -m \\
0 & 1\end{pmatrix}$. Note that $C=(AB)^{-1}$ has negative trace.
       \begin{center}
\begin{tikzpicture}[x=1cm,y=1cm]


\draw [xscale=cos(70),  thick] (3,0.2) arc (-90:90:0.3);
\draw [xscale=cos(70), thick] (3,0.8) arc (90:270:0.3);


\draw    (-0.5, 0.2) node{{\tiny $\rro=-1$}};
\draw    (1, 1) node{{\tiny $t<0$}};
\draw    (1, -1) node {{\tiny $\rro=-1$}};
\draw    (0.5, 0) node{{\tiny 0}};

\draw[thick] (0,0) -- (1,0.8);
\draw[thick] (0,0)-- (1, -0.8);
\draw[thick] (1, 0.25)--(1, -0.8);

\draw [xscale=cos(70),  thick] (3,0.2) arc (-90:90:0.3);
\draw [xscale=cos(70), thick] (3,0.8) arc (90:270:0.3);

\end{tikzpicture}
\end{center} 
We can further deform the previous one so that $C$ becomes a parabolic element with negative trace $2+mn=-2$. Explicitly, one can take $n=1, m=-4$.
     \begin{center}
\begin{tikzpicture}[x=1cm,y=1cm]




\draw    (-0.5, 0.2) node{{\tiny $\rro=-1$}};
\draw    (1, 1) node{{\tiny $\rro=0$}};
\draw    (1, -1) node {{\tiny $\rro=-1$}};
\draw    (0.5, 0) node{{\tiny 0}};

\draw[thick] (0,0) -- (1,0.8);
\draw[thick] (0,0)-- (1, -0.8);
\draw[thick] (1, 0.8)--(1, -0.8);
\end{tikzpicture}
\end{center} 
\subsection{Punctured torus}\label{pun-torus}
The punctured torus' fundamental group admits the presentation $\langle A, B: [A,B]C=id \rangle$.

\smallskip

$\bullet$  Signature 0 blocks\\

In Goldman's paper \cite{Gold1}, Lemma 7.2 explains that there exist pairs $(A,B)\in 
Hyp_0\cup Par_0$ (hence traces are positive) whose commutator $[A,B]\in Hyp_0$. It follows that $tr([A,B]^{-1})>0$. Taking $C=[A,B]^{-1}$, the Toledo invariant of the representation determined by $A,B$ is zero.
\begin{center}
\begin{tikzpicture}[x=1cm,y=1cm]

\begin{scope}[shift={(2,0)}, thick]
\clip(-1.8,-2)rectangle(3,2);
\draw (0,0) circle [x radius=2, y radius=1];
\end{scope}

\draw[shift={(2,0)}, yscale=cos(70), thick] (-1,0) arc (-180:0:1);
\draw[shift={(2,0)}, yscale=cos(70), thick] (-0.7,-0.6) arc (180:0:0.7);

\draw [thick] (0,0.3) .. controls (0.1,0.3) and (0.1,0.34) .. (0.21,0.446);
\draw [thick] (0,-0.3) .. controls (0.1,-0.3) and (0.1,-0.34) .. (0.21,-0.446);

\draw [xscale=cos(70), dashed, thick] (0,-0.3) arc (-90:90:0.3);
\draw [xscale=cos(70), thick] (0,0.3) arc (90:270:0.3);

\draw (-0.1,-0.7) node{{$t>0$}};
\draw  (0.5, 0) node{{ 0}};
\end{tikzpicture}
\end{center}

$\bullet$ Signature 1 blocks\\

Take $A=\begin{pmatrix}
           \lambda & 0\\
           0             & \lambda^{-1}\end{pmatrix},\lambda>1, B=\begin{pmatrix}
                                                                                \lambda^{-1} & c\\
                                                                                0         &    \lambda\end{pmatrix}$ so that
 $[A,B]=\begin{pmatrix}
                1 & c(\lambda-\lambda^{-1}) \\
                 0  & 1\end{pmatrix}$.
    This representation has image inside the Borel subgroup, hence its Toledo invariant is zero.  We take $C=[A,B]^{-1}$. Depending on the sign of $c$, $\text{sign}(\phi)=1, \rro(C)=1$ if $c>0$, or $\text{sign}(\phi)=-1,\rro(C)=-1$ if $c<0$.
                 
\begin{center}
\begin{tikzpicture}[x=1cm,y=1cm]

\begin{scope}[shift={(2,0)}, thick]
\clip(-1.8,-2)rectangle(3,2);
\draw (0,0) circle [x radius=2, y radius=1];
\end{scope}

\draw[shift={(2,0)}, yscale=cos(70), thick] (-1,0) arc (-180:0:1);
\draw[shift={(2,0)}, yscale=cos(70), thick] (-0.7,-0.6) arc (180:0:0.7);

\draw [thick] (0,0.3) .. controls (0.1,0.3) and (0.1,0.34) .. (0.21,0.446);
\draw [thick] (0,-0.3) .. controls (0.1,-0.3) and (0.1,-0.34) .. (0.21,-0.446);


\draw[thick] (-0.5,0) -- (0.03,0.28);
\draw[thick] (-0.5,0)-- (0.03, -0.28);

\draw (-0.3,-0.6) node{{$\rro=1$}};
\draw  (0.5, 0) node{{ 1}};

\begin{scope}[shift={(8,0)}, thick]
\clip(-1.8,-2)rectangle(3,2);
\draw (0,0) circle [x radius=2, y radius=1];
\end{scope}

\draw[shift={(8,0)}, yscale=cos(70), thick] (-1,0) arc (-180:0:1);
\draw[shift={(8,0)}, yscale=cos(70), thick] (-0.7,-0.6) arc (180:0:0.7);

\draw [shift={(6,0)}, thick] (0,0.3) .. controls (0.1,0.3) and (0.1,0.34) .. (0.21,0.446);
\draw [shift={(6,0)},thick] (0,-0.3) .. controls (0.1,-0.3) and (0.1,-0.34) .. (0.21,-0.446);


\draw[shift={(6,0)}, thick] (-0.5,0) -- (0.03,0.28);
\draw[shift={(6,0)},thick] (-0.5,0)-- (0.03, -0.28);

\draw [shift={(5.8,0)}](-0.3,-0.6) node{{$\rro=-1$}};
\draw [shift={(6,0)}]  (0.5, 0) node{{ $-1$}};

\end{tikzpicture}
\end{center}

$\bullet$ Signature 2 blocks\\

If one takes a hyperbolic structure on the punctured torus with geodesic boundary, the Euler class has maximum value 1 \cite[Theorem D]{Gold1}, hence the Toledo invariant is also 1. By (\ref{relation}), for $C'$ with positive trace, $[A,B]C'=-I$. Then $[A,B]\in Hyp_1$ has a negative trace. Hence $C=[A,B]^{-1}$ has a negative trace. 

If one deforms the geodesic boundary to a parabolic one, the Toledo invariant is still maximum 1, but $\rro(C)=0$ since $C$ has a negative trace. 

\begin{center}
\begin{tikzpicture}[x=1cm,y=1cm]

\begin{scope}[shift={(2,0)}, thick]
\clip(-1.8,-2)rectangle(3,2);
\draw (0,0) circle [x radius=2, y radius=1];
\end{scope}

\draw[shift={(2,0)}, yscale=cos(70), thick] (-1,0) arc (-180:0:1);
\draw[shift={(2,0)}, yscale=cos(70), thick] (-0.7,-0.6) arc (180:0:0.7);

\draw [thick] (0,0.3) .. controls (0.1,0.3) and (0.1,0.34) .. (0.21,0.446);
\draw [thick] (0,-0.3) .. controls (0.1,-0.3) and (0.1,-0.34) .. (0.21,-0.446);

\draw [xscale=cos(70), dashed, thick] (0,-0.3) arc (-90:90:0.3);
\draw [xscale=cos(70), thick] (0,0.3) arc (90:270:0.3);


\draw (-0.3,-0.6) node{{$t<0$}};
\draw  (0.5, 0) node{{ 2}};

\begin{scope}[shift={(8,0)}, thick]
\clip(-1.8,-2)rectangle(3,2);
\draw (0,0) circle [x radius=2, y radius=1];
\end{scope}

\draw[shift={(8,0)}, yscale=cos(70), thick] (-1,0) arc (-180:0:1);
\draw[shift={(8,0)}, yscale=cos(70), thick] (-0.7,-0.6) arc (180:0:0.7);

\draw [shift={(6,0)}, thick] (0,0.3) .. controls (0.1,0.3) and (0.1,0.34) .. (0.21,0.446);
\draw [shift={(6,0)},thick] (0,-0.3) .. controls (0.1,-0.3) and (0.1,-0.34) .. (0.21,-0.446);


\draw[shift={(6,0)}, thick] (-0.5,0) -- (0.03,0.28);
\draw[shift={(6,0)},thick] (-0.5,0)-- (0.03, -0.28);

\draw [shift={(5.8,0)}](-0.3,-0.6) node{{$\rro=0$}};
\draw [shift={(6,0)}]  (0.5, 0) node{{ $2$}};

\end{tikzpicture}
\end{center}

\section{Genus 0 case}

Note that for a pair of pants $P$, there is a building block of all signatures 0, 1, 2 with
a positive trace hyperbolic element on the boundary.

Let us prove by induction on $k$ the following statement: \emph{every integer between $0$ and $2k$ is equal to the signature of a flat bundle over a genus $0$ surface, in such a way that
\begin{itemize}
  \item the boundary holonomy is hyperbolic or parabolic,
  \item there is at least one boundary component along which the holonomy is hyperbolic with positive trace.
\end{itemize}}

The case $k=1$ has been covered in the previous section. 

Here comes the inductive step. Let $m=0,\ldots,2k$. By induction, we can start from a genus zero surface $\Sigma$ such that $|\chi(\Sigma)|=k$ equipped with a flat bundle of signature $m$. By induction again, the boundary holonomy is hyperbolic or parabolic, and one of the boundary components of $\Sigma$ has a holonomy which is hyperbolic with positive trace. If we attach
to this boundary one of the 3 types of building blocks described in the previous section, along a boundary with hyperbolic holonomy of positive trace, as in the following pictures, we can achieve signatures $m$, $m+1$ and $m+2$. 

Building block suitable to achieve signature $m$:
 \begin{center}
\begin{tikzpicture}[x=1cm,y=1cm]



\draw[thick] (0,0.3) -- (1,0.8);
\draw[thick] (0,-0.3)-- (1, -0.8);
\draw[thick] (1, 0.25)--(1, -0.25);

\draw [xscale=cos(70),  thick] (0,-0.3) arc (-90:90:0.3);
\draw [xscale=cos(70), thick] (0,0.3) arc (90:270:0.3);

\draw [xscale=cos(70),  thick] (3,0.2) arc (-90:90:0.3);
\draw [xscale=cos(70), thick] (3,0.8) arc (90:270:0.3);

\draw [xscale=cos(70),  thick] (3,-0.8) arc (-90:90:0.3);
\draw [xscale=cos(70), thick] (3,-0.2) arc (90:270:0.3);



\draw    (0, 0.6) node{{\tiny $t>0$}};
\draw    (1, 1) node{{\tiny $t>0$}};
\draw    (1, -1) node {{\tiny $t>0$}};
\draw    (0.5, 0) node{{\tiny 0}};

\draw[thick] (0,0.3) -- (1,0.8);
\draw[thick] (0,-0.3)-- (1, -0.8);
\draw[thick] (1, 0.25)--(1, -0.25);

\end{tikzpicture}
\end{center}

Building block suitable to achieve signature $m+1$:
  \begin{center}
\begin{tikzpicture}[x=1cm,y=1cm]

\draw[thick] (0,0.3) -- (1,0.8);
\draw[thick] (0,-0.3)-- (1, -0.8);
\draw[thick] (1, 0.25)--(1, -0.25);

\draw [xscale=cos(70),  thick] (0,-0.3) arc (-90:90:0.3);
\draw [xscale=cos(70), thick] (0,0.3) arc (90:270:0.3);


\draw [xscale=cos(70),  thick] (3,-0.8) arc (-90:90:0.3);
\draw [xscale=cos(70), thick] (3,-0.2) arc (90:270:0.3);

\draw    (0, 0.6) node{{\tiny $t>0$}};
\draw    (1, 1) node{{\tiny $\rro=1$}};
\draw    (1, -1) node {{\tiny $t>0$}};
\draw    (0.5, 0) node{{\tiny 1}};

\draw[thick] (0,0.3) -- (1,0.8);
\draw[thick] (0,-0.3)-- (1, -0.8);
\draw[thick] (1, 0.25)--(1, 0.8);

\end{tikzpicture}
\end{center}   

Building block suitable to achieve signature $m+2$:
\begin{center}
\begin{tikzpicture}[x=1cm,y=1cm]



\draw[thick] (0,0.3) -- (1,0.8);
\draw[thick] (0,-0.3)-- (1, -0.8);
\draw[thick] (1, 0.25)--(1, -0.25);

\draw [xscale=cos(70),  thick] (0,-0.3) arc (-90:90:0.3);
\draw [xscale=cos(70), thick] (0,0.3) arc (90:270:0.3);

\draw [xscale=cos(70),  thick] (3,0.2) arc (-90:90:0.3);
\draw [xscale=cos(70), thick] (3,0.8) arc (90:270:0.3);

\draw [xscale=cos(70),  thick] (3,-0.8) arc (-90:90:0.3);
\draw [xscale=cos(70), thick] (3,-0.2) arc (90:270:0.3);



\draw    (0, 0.6) node{{\tiny $t>0$}};
\draw    (1, 1) node{{\tiny $t<0$}};
\draw    (1, -1) node {{\tiny $t>0$}};
\draw    (0.5, 0) node{{\tiny 2}};
\end{tikzpicture}
\end{center}

Therefore the new surface has genus $0$, Euler characteristic $-k-1$, it admits flat bundles whose signatures achieve all values from $0$ to $2k+2$, while satisfying the boundary requirements. Here is an example:

 \begin{center}
\begin{tikzpicture}[x=1cm,y=1cm]



\draw[thick] (0,0.3) -- (1,0.8);
\draw[thick] (0,-0.3)-- (1, -0.8);
\draw[thick] (1, 0.25)--(1, -0.25);

\draw [xscale=cos(70),  thick] (0,-0.3) arc (-90:90:0.3);
\draw [xscale=cos(70), thick] (0,0.3) arc (90:270:0.3);

\draw [xscale=cos(70),  thick] (3,0.2) arc (-90:90:0.3);
\draw [xscale=cos(70), thick] (3,0.8) arc (90:270:0.3);

\draw [xscale=cos(70),  thick] (3,-0.8) arc (-90:90:0.3);
\draw [xscale=cos(70), thick] (3,-0.2) arc (90:270:0.3);



\draw    (0, 0.6) node{{\tiny $t>0$}};
\draw    (1, 1) node{{\tiny $t<0$}};
\draw    (1, -1.1) node {{\tiny $t>0$}};
\draw    (0.5, 0) node{{\tiny 2}};

\draw[thick] (0,0.3) -- (1,0.8);
\draw[thick] (0,-0.3)-- (1, -0.8);
\draw[thick] (1, 0.25)--(1, -0.25);

\draw [shift={(1,-0.5)}, xscale=cos(70),  thick] (0,-0.3) arc (-90:90:0.3);
\draw [shift={(1,-0.5)}, xscale=cos(70), thick] (0,0.3) arc (90:270:0.3);


\draw [shift={(1,-0.5)}, xscale=cos(70),  thick] (3,-0.8) arc (-90:90:0.3);
\draw [shift={(1,-0.5)}, xscale=cos(70), thick] (3,-0.2) arc (90:270:0.3);

\draw    [shift={(1,-0.5)}](1, 1) node{{\tiny $\rro=1$}};
\draw    [shift={(1,-0.5)}](1, -1.1) node {{\tiny $t>0$}};
\draw   [shift={(1,-0.5)}] (0.5, 0) node{{\tiny 1}};

\draw[shift={(1,-0.5)},thick] (0,0.3) -- (1,0.8);
\draw[shift={(1,-0.5)},thick] (0,-0.3)-- (1, -0.8);
\draw[shift={(1,0.1)},thick] (1, 0.25)--(1, -0.8);

\draw[thick] (0,0.3) -- (1,0.8);
\draw[thick] (0,-0.3)-- (1, -0.8);
\draw[thick] (1, 0.25)--(1, -0.25);

\draw [shift={(2,-1)}, xscale=cos(70),  thick] (0,-0.3) arc (-90:90:0.3);
\draw [shift={(2,-1)}, xscale=cos(70), thick] (0,0.3) arc (90:270:0.3);

\draw [shift={(2,-1)}, xscale=cos(70),  thick] (3,0.2) arc (-90:90:0.3);
\draw [shift={(2,-1)}, xscale=cos(70), thick] (3,0.8) arc (90:270:0.3);

\draw [shift={(2,-1)}, xscale=cos(70),  thick] (3,-0.8) arc (-90:90:0.3);
\draw [shift={(2,-1)}, xscale=cos(70), thick] (3,-0.2) arc (90:270:0.3);

\draw    [shift={(2,-1)}](1, 1) node{{\tiny $t>0$}};
\draw    [shift={(2,-1)}](1, -1) node {{\tiny $t>0$}};
\draw    [shift={(2,-1)}](0.5, 0) node{{\tiny 0}};

\draw[shift={(2,-1)}, thick] (0,0.3) -- (1,0.8);
\draw[shift={(2,-1)}, thick] (0,-0.3)-- (1, -0.8);
\draw[shift={(2,-1)}, thick] (1, 0.25)--(1, -0.25);
\end{tikzpicture}
\end{center}   

This concludes the induction proof.

\section{Genus 1 case}

Let $\Sigma$ be a twice punctured torus. Let us decompose $\Sigma$ into a punctured torus and a pair of pants. The five pictures below show how to glue together building blocks in order to achieve all integers up to 4 as signatures of flat bundles with hyperbolic or parabolic boundary holonomy, and at least one boundary component whose holonomy is hyperbolic with positive trace.

\begin{center}
\begin{tikzpicture}[x=1cm,y=1cm]

\begin{scope}[shift={(2,0)}, thick]
\clip(-1.8,-2)rectangle(3,2);
\draw (0,0) circle [x radius=2, y radius=1];
\end{scope}

\draw[shift={(2,0)}, yscale=cos(70), thick] (-1,0) arc (-180:0:1);
\draw[shift={(2,0)}, yscale=cos(70), thick] (-0.7,-0.6) arc (180:0:0.7);

\draw [thick] (0,0.3) .. controls (0.1,0.3) and (0.1,0.34) .. (0.21,0.446);
\draw [thick] (0,-0.3) .. controls (0.1,-0.3) and (0.1,-0.34) .. (0.21,-0.446);

\draw [xscale=cos(70), dashed, thick] (0,-0.3) arc (-90:90:0.3);
\draw [xscale=cos(70), thick] (0,0.3) arc (90:270:0.3);

\draw  (0.5, 0) node{{ 0}};

\draw [shift={(-1,0.5)}, xscale=cos(70),  thick] (0,-0.3) arc (-90:90:0.3);
\draw [shift={(-1,0.5)}, xscale=cos(70), thick] (0,0.3) arc (90:270:0.3);

\draw [shift={(-1,0.5)}, xscale=cos(70),  thick] (3,-0.8) arc (-90:90:0.3);
\draw [shift={(-1,0.5)}, xscale=cos(70), thick] (3,-0.2) arc (90:270:0.3);

\draw    [shift={(-1,0.5)}](0, 0.6) node{{\tiny $t>0$}};
\draw    [shift={(-1,0.5)}](1, 1) node{{\tiny $t>0$}};
\draw    [shift={(-1,0.5)}](1, -1.2) node {{\tiny $t>0$}};
\draw    [shift={(-1,0.5)}](0.5, 0) node{{\tiny 0}};

\draw[shift={(-1,0.5)}, thick] (0,0.3) -- (1,0.8);
\draw[shift={(-1,0.5)}, thick] (0,-0.3)-- (1, -0.8);
\draw[shift={(-1,0.5)}, thick] (1, -0.2)--(1, 0.8);
\end{tikzpicture}
\end{center}

\begin{center}
\begin{tikzpicture}[x=1cm,y=1cm]

\begin{scope}[shift={(2,0)}, thick]
\clip(-1.8,-2)rectangle(3,2);
\draw (0,0) circle [x radius=2, y radius=1];
\end{scope}

\draw[shift={(2,0)}, yscale=cos(70), thick] (-1,0) arc (-180:0:1);
\draw[shift={(2,0)}, yscale=cos(70), thick] (-0.7,-0.6) arc (180:0:0.7);

\draw [thick] (0,0.3) .. controls (0.1,0.3) and (0.1,0.34) .. (0.21,0.446);
\draw [thick] (0,-0.3) .. controls (0.1,-0.3) and (0.1,-0.34) .. (0.21,-0.446);

\draw [xscale=cos(70), dashed, thick] (0,-0.3) arc (-90:90:0.3);
\draw [xscale=cos(70), thick] (0,0.3) arc (90:270:0.3);

\draw  (0.5, 0) node{{ 0}};

\draw [shift={(-1,0.5)}, xscale=cos(70),  thick] (0,-0.3) arc (-90:90:0.3);
\draw [shift={(-1,0.5)}, xscale=cos(70), thick] (0,0.3) arc (90:270:0.3);

\draw [shift={(-1,0.5)}, xscale=cos(70),  thick] (3,-0.8) arc (-90:90:0.3);
\draw [shift={(-1,0.5)}, xscale=cos(70), thick] (3,-0.2) arc (90:270:0.3);

\draw    [shift={(-1,0.5)}](0, 0.6) node{{\tiny $t>0$}};
\draw    [shift={(-1,0.5)}](1, 1) node{{\tiny $\rro=1$}};
\draw    [shift={(-1,0.5)}](1, -1.15) node {{\tiny $t>0$}};
\draw    [shift={(-1,0.5)}](0.5, 0) node{{\tiny 1}};

\draw[shift={(-1,0.5)}, thick] (0,0.3) -- (1,0.8);
\draw[shift={(-1,0.5)}, thick] (0,-0.3)-- (1, -0.8);
\draw[shift={(-1,0.5)}, thick] (1, -0.2)--(1, 0.8);
\end{tikzpicture}
\end{center}

\begin{center}
\begin{tikzpicture}[x=1cm,y=1cm]

\begin{scope}[shift={(2,0)}, thick]
\clip(-1.8,-2)rectangle(3,2);
\draw (0,0) circle [x radius=2, y radius=1];
\end{scope}

\draw[shift={(2,0)}, yscale=cos(70), thick] (-1,0) arc (-180:0:1);
\draw[shift={(2,0)}, yscale=cos(70), thick] (-0.7,-0.6) arc (180:0:0.7);

\draw [thick] (0,0.3) .. controls (0.1,0.3) and (0.1,0.34) .. (0.21,0.446);
\draw [thick] (0,-0.3) .. controls (0.1,-0.3) and (0.1,-0.34) .. (0.21,-0.446);

\draw [xscale=cos(70), dashed, thick] (0,-0.3) arc (-90:90:0.3);
\draw [xscale=cos(70), thick] (0,0.3) arc (90:270:0.3);

\draw  (0.5, 0) node{{ 0}};

\draw [shift={(-1,0.5)}, xscale=cos(70),  thick] (0,-0.3) arc (-90:90:0.3);
\draw [shift={(-1,0.5)}, xscale=cos(70), thick] (0,0.3) arc (90:270:0.3);

\draw [shift={(-1,0.5)}, xscale=cos(70),  thick] (3,-0.8) arc (-90:90:0.3);
\draw [shift={(-1,0.5)}, xscale=cos(70), thick] (3,-0.2) arc (90:270:0.3);

\draw    [shift={(-1,0.5)}](0, 0.6) node{{\tiny $t<0$}};
\draw    [shift={(-1,0.5)}](1, 1) node{{\tiny $t>0$}};
\draw    [shift={(-1,0.5)}](1, -1.2) node {{\tiny $t>0$}};
\draw    [shift={(-1,0.5)}](0.5, 0) node{{\tiny 2}};

\draw[shift={(-1,0.5)}, thick] (0,0.3) -- (1,0.8);
\draw[shift={(-1,0.5)}, thick] (0,-0.3)-- (1, -0.8);
\draw[shift={(-1,0.5)}, thick] (1, -0.2)--(1, 0.8);
\end{tikzpicture}
\end{center}

\begin{center}
\begin{tikzpicture}[x=1cm,y=1cm]

\begin{scope}[shift={(2,0)}, thick]
\clip(-1.8,-2)rectangle(3,2);
\draw (0,0) circle [x radius=2, y radius=1];
\end{scope}

\draw[shift={(2,0)}, yscale=cos(70), thick] (-1,0) arc (-180:0:1);
\draw[shift={(2,0)}, yscale=cos(70), thick] (-0.7,-0.6) arc (180:0:0.7);

\draw [thick] (0,0.3) .. controls (0.1,0.3) and (0.1,0.34) .. (0.21,0.446);
\draw [thick] (0,-0.3) .. controls (0.1,-0.3) and (0.1,-0.34) .. (0.21,-0.446);

\draw [xscale=cos(70), dashed, thick] (0,-0.3) arc (-90:90:0.3);
\draw [xscale=cos(70), thick] (0,0.3) arc (90:270:0.3);


\draw  (0.5, 0) node{{ 2}};

\draw [shift={(-1,-0.5)}, xscale=cos(70),  thick] (3,0.2) arc (-90:90:0.3);
\draw [shift={(-1,-0.5)}, xscale=cos(70), thick] (3,0.8) arc (90:270:0.3);

\draw [shift={(-1,-0.5)}, xscale=cos(70),  thick] (3,-0.8) arc (-90:90:0.3);
\draw [shift={(-1,-0.5)}, xscale=cos(70), thick] (3,-0.2) arc (90:270:0.3);

\draw    [shift={(-1,-0.5)}](-0.5, 0.2) node{{\tiny $\rro=-1$}};
\draw    [shift={(-1,-0.5)}](1, 1.15) node{{\tiny $t<0$}};
\draw    [shift={(-1,-0.5)}](1, -1) node {{\tiny $t>0$}};
\draw    [shift={(-1,-0.5)}](0.5, 0) node{{\tiny 1}};

\draw[shift={(-1,-0.5)},thick] (0,0) -- (1,0.8);
\draw[shift={(-1,-0.5)},thick] (0,0)-- (1, -0.8);
\draw[shift={(-1,-0.5)},thick] (1, 0.25)--(1, -0.25);

\end{tikzpicture}
\end{center}

\begin{center}
\begin{tikzpicture}[x=1cm,y=1cm]

\begin{scope}[shift={(2,0)}, thick]
\clip(-1.8,-2)rectangle(3,2);
\draw (0,0) circle [x radius=2, y radius=1];
\end{scope}

\draw[shift={(2,0)}, yscale=cos(70), thick] (-1,0) arc (-180:0:1);
\draw[shift={(2,0)}, yscale=cos(70), thick] (-0.7,-0.6) arc (180:0:0.7);

\draw [thick] (0,0.3) .. controls (0.1,0.3) and (0.1,0.34) .. (0.21,0.446);
\draw [thick] (0,-0.3) .. controls (0.1,-0.3) and (0.1,-0.34) .. (0.21,-0.446);

\draw [xscale=cos(70), dashed, thick] (0,-0.3) arc (-90:90:0.3);
\draw [xscale=cos(70), thick] (0,0.3) arc (90:270:0.3);


\draw  (0.5, 0) node{{ 2}};

\draw[shift={(-1,-0.5)},thick] (0,0.3) -- (1,0.8);
\draw[shift={(-1,-0.5)},thick] (0,-0.3)-- (1, -0.8);
\draw[shift={(-1,-0.5)},thick] (1, 0.25)--(1, -0.25);

\draw [shift={(-1,-0.5)},xscale=cos(70),  thick] (0,-0.3) arc (-90:90:0.3);
\draw [shift={(-1,-0.5)},xscale=cos(70), thick] (0,0.3) arc (90:270:0.3);

\draw [shift={(-1,-0.5)},xscale=cos(70),  thick] (3,0.2) arc (-90:90:0.3);
\draw [shift={(-1,-0.5)},xscale=cos(70), thick] (3,0.8) arc (90:270:0.3);

\draw [shift={(-1,-0.5)},xscale=cos(70),  thick] (3,-0.8) arc (-90:90:0.3);
\draw [shift={(-1,-0.5)},xscale=cos(70), thick] (3,-0.2) arc (90:270:0.3);

\draw    [shift={(-1,-0.5)}](0, 0.6) node{{\tiny $t>0$}};
\draw    [shift={(-1,-0.5)}](1, 1.15) node{{\tiny $t<0$}};
\draw    [shift={(-1,-0.5)}](1, -1) node {{\tiny $t>0$}};
\draw    [shift={(-1,-0.5)}](0.5, 0) node{{\tiny 2}};
\end{tikzpicture}
\end{center}

Since there is always a free $t>0$ boundary on the last pair of pants, we can do the same induction as in the genus 0 case to attach one more pair of pants.

\medskip

We conclude this discussion of low genus surfaces. This completes the proof of Theorem \ref{main}.

\begin{prop}\label{zero}
For every positive integer $k$, every integer $m$ such that $-2k\leq m \leq 2k$, and for each genus 0 and 1, there exists a punctured surface $\Sigma$ of the required genus with $\chi(\Sigma)=-k$ and a representation $\phi:\pi_1(\Sigma)\ra Sp(2,\br)$ which is hyperbolic or parabolic on the boundary and whose signature is $m$.
\end{prop}

\begin{proof}
By composing with an involution of $Sp(2,\br)$ (Formula \ref{anti-holomorphic}), it suffices to realize the positive numbers $m$.
By the previous section, for genus 0 or 1 surfaces $\Sigma$ with $\chi(\Sigma)=-k$, every possible value for signature between 0 and $2k$ is attained by
some representation $\phi:\pi_1(\Sigma)\ra Sp(2,\br)$. By construction, the boundary holonomy is always hyperbolic or parabolic.
\end{proof}




 

\section{Boundary elliptic representations in $SL(2,\br)$}
\label{boundaryelliptic}

\subsection{Scheme of the proof of Theorem \ref{elliptic}}

Signatures of boundary elliptic representations are always even, Proposition \ref{even}. The higher genus case is handled in Proposition \ref{elliptichigher}. Then we study the remainder mod $4$ of the signature in genus $0$. Both fully boundary elliptic representations and representations allowing one hyperbolic boundary holonomy are described in Corollaries \ref{ellipticgenus0} and \ref{ellipticgenus0}. The latter case is used to study genus one surfaces, viewed as connected sums of a genus $0$ surface and a punctured torus, in Corollary \ref{elliptgenus1}.

\subsection{A necessary condition}

Say an element of $SL(2,\br)$ is \emph{strictly parabolic} if it is parabolic with trace equal to $2$. In other words, one excludes parabolics with eigenvalues $-1$.

Let $\Sigma$ be an oriented surface with boundary. Say a representation $\pi_1(\Sigma)\ra SL(2,\br))$ is 
\begin{itemize}
  \item \emph{boundary non-s-parabolic} if its values on boundary components are never strictly parabolic elements of $SL(2,\br)$.
  \item \emph{boundary elliptic} if its values on boundary components are all elliptic elements of $SL(2,\br)$.
\end{itemize}

\begin{prop}\label{even}
Let $\Sigma$ be an oriented surface with boundary. Boundary non-s-parabolic representations $\pi_1(\Sigma)\ra SL
(2,\br)$ have even signatures.
\end{prop}

On $G=SL(2,\br)$, the rho invariant is discontinuous only at $I$ and at strictly parabolic elements. Furthermore, the rho invariant mod $2\bz$ is continuous on the union of $\{I\}$ and of elliptic, hyperbolic and nonstrictly parabolic elements. Since the Toledo invariant is continuous, the signature mod $2$ is locally constant on the set of boundary non-s-parabolic representations. We need work a bit more to show that it is actually constant.

Let $\mathcal{P}\subset G^{n}$ denote the set of $n$-tuples where at least one entry is strictly parabolic. The closure $\bar{\mathcal{P}}=\{\exists k=1,\ldots,n,\,tr(C_k)=2\}$ is a real analytic subset of $G^n$. Using the free generators $(A_1,B_1,\ldots,A_g,B_g,C_1,\ldots,C_{n-1})$, let us view $Hom(\pi_1(\Sigma),G)$ as $G^{2g+n-1}$. The value of the representation on boundary components is given by the real analytic map
$$
F(A_1,B_1,\ldots,A_g,B_g,C_1,\ldots,C_{n-1})=(C_1,\ldots,C_{n-1},(\prod_{i=1}^{g} [A_i, B_i]\prod_{k=1}^{n-1} C_k)^{-1}).
$$

Let $\lambda:\br\to G^{2g+n-1}$ be a real analytic arc which is not contained in $F^{-1}(\bar{\mathcal{P}})$. Let us show that sign mod $2$ is constant on $[0,1]\setminus (F\circ\lambda)^{-1}(\mathcal{P})$. 
The real analytic subset $J=\{0,1\}\cup([0,1]\cap(F\circ\lambda)^{-1}(\bar{\mathcal{P}}))$ is finite. 

Let $a<b<c$ be three consecutive elements of $J$. On each of $(a,b)$ and $(b,c)$, the holonomy of each boundary component is either elliptic or hyperbolic or nonstrictly parabolic, so its rho invariant varies continuously. An elliptic holonomy is a rotation whose angle tends either to $0$ or $2\pi$ as $t$ tends to $b$, so its rho invariant tends either to $2$ or to $-2$. Thus, when traversing $b$, the rho invariant of each boundary component jumps by some number in $\{-4,-2,0,2,4\}$. The whole rho invariant of the boundary $\rro(\p\Sigma)$ jumps by an even integer. Since the Toledo term is continuous, the signature mod $2$ does not change when one passes from $(a,b)$ to $(b,c)$, i.e., it is constant on $(a,b)\cup(b,c)$. 

Let $a<b$ be two consecutive elements of $J$. If $b\in (F\circ\lambda)^{-1}(\bar{\mathcal{P}}\setminus\mathcal{P})$, that is the holonomy of some boundary component is equal to $I$, the rho invariant of this component again possibly jumps by $\pm 2$ when passing from $(a,b)$ to $b$. So, in this case, the signature mod $2$ is constant on $(a,b]$. This completes the proof that sign mod $2$ is constant on $[0,1]\setminus (F\circ\lambda)^{-1}(\mathcal{P})$. 

Since $F^{-1}(\bar{\mathcal{P}})$ has positive codimension, every pair of points of $G^{2g+n-1}$ can be joined by a real analytic arc which is not contained in $F^{-1}(\bar{\mathcal{P}})$. This shows that signature mod $2$ is constant on the complement of $F^{-1}(\mathcal{P})$, i.e. on the set of boundary non-s-parabolic representations. The trivial representation is non-s-parabolic, its Toledo and rho invariants both vanish, so its signature vanishes. One concludes that signatures of non-s-parabolic representations are equal to $0$ mod $2$.

\subsection{Higher genus}

\begin{prop}\label{elliptichigher}
Let $\Sigma$ be an oriented surface with boundary, of genus $>1$. Every even integer between $2\chi(\Sigma)$ and $-2\chi(\Sigma)$ is the signature of a boundary elliptic representation $\pi_1(\Sigma)\ra SL(2,\br)$.
\end{prop}

Let $G=SL(2,\br)$. We view $\Sigma$ as obtained by removing $n$ disjoint disks from a closed orientable surface $S$ of genus $>1$.
As in Section 3, we can start with a representation $\phi\in\text{Hom}(\pi_1(S), G)$ which is a regular point of the map
$$
\Phi:G^{2g}\ra G,\ \Phi(A_1,B_1,\ldots,A_g,B_g)=\prod_{i=1}^{g} [A_i, B_i].
$$
Denote by $c_1,\ldots,c_n$ the boundary components of $\Sigma$, $n\geq 1$. Let $\psi_0\in\text{Hom}(\pi_1(\Sigma), G)$ denote the representation such that 
\begin{align*}
\forall i=1,\ldots,g,\quad & \psi_0(A_i)=\phi(A_i),\quad \psi_0(B_i)=\phi(B_i),\\
\forall k=1,\ldots,n,\quad & \psi_0(C_k)=I.
\end{align*}
Let us pick continuous functions $\theta_1,\ldots,\theta_n:\br_+\ra[0,2\pi]$, such that 
$$
\sum_{k=1}^n \theta_k(0)=2a\pi
$$
for some integer $a$ between $0$ and $n$, and $\theta_k(t)\in(0,2\pi)$ for $t>0$. For $k=1,\ldots,n$, let
\begin{equation}\label{eqn7}
  \psi_t(c_k)\sim \begin{pmatrix}
  \cos\theta_k(t)&-\sin\theta_k(t) \\
  \sin\theta_k(t)&\cos\theta_k(t)
\end{pmatrix}.
\end{equation}
By construction, $\Phi$ is a submersion near $(\phi(A_1),\phi(B_1),\cdots,\phi(A_g),\phi(B_g))$, so there exists a continuous lift
$$
t\mapsto (\psi_t(A_1),\psi_t(B_1),\cdots,\psi_t(A_g),\psi_t(B_g))
$$ 
of
$$
t\mapsto\left(\prod_{k=1}^{n} \psi_t(c_k)\right)^{-1}
$$
under $\Phi$. This produces a continuous family of representations $\psi_t\in\text{Hom}(\pi_1(\Sigma), G)$ which are elliptic on the boundary.

Since, for $t>0$,
\begin{equation*}
  \rro_{\psi_t}(\p\Sigma)=\sum_{k=1}^n 2(\mathrm{sgn}(\theta_k(t))-\frac{\theta_k(t)}{\pi})=2n-\frac{2}{\pi}\sum_{k=1}^n \theta_k(t),
\end{equation*}
$$
\lim_{t\to 0} \rro_{\psi_t}(\p\Sigma)=2n-4a.
$$
Using \cite{KPW}'s formula
\begin{equation*}
   \mr{sign}(\psi_t)=2\text{T}(\psi_t)+\rro_{\psi_t}(\p\Sigma),
\end{equation*}
the continuity of the Toledo invariant implies that
\begin{align*}
\lim_{t\to 0}  \mr{sign}(\psi_t)=2\text{T}(\psi_0)+2n-4a.
\end{align*}
Since, according to Goldman, 
\begin{align*}
\text{T}(\psi_0)=\text{T}(\phi)=\text{Euler}(\phi)
\end{align*}
takes every integer value between $\chi(S)$ and $-\chi(S)$, for $t$ small, $\mr{sign}(\psi_t)$ takes all values of the form
$$
2(n+b-2a), \quad\text{for integers }a\in\{0,\ldots,n\}\text{ and }b\in[\chi(S),-\chi(S)].
$$
Obviously, taking $a=0$ and letting $b$ vary between $\chi(S)$ and $-\chi(S)$, $n+b-2a$ covers the interval $[n-|\chi(S)|,n+|\chi(S)|]$. Since $\chi(S)<0$, the length of this interval is at least $3$. Subtracting $2,4,6,\ldots$ yields overlapping intervals whose union covers $[n-|\chi(S)|-2,n+|\chi(S)|]$, $[n-|\chi(S)|-4,n+|\chi(S)|]$,...,$[n-|\chi(S)|-2n,n+|\chi(S)|]=[\chi(\Sigma),-\chi(\Sigma)]$. Hence $2(n+b-2a)$ takes every even value in $[2\chi(\Sigma),-2\chi(\Sigma)]$.

\subsection{Genus 0}

\begin{prop}\label{deformation}
Let $\Sigma$ be an oriented surface with boundary, of genus $0$. Let $\phi_0:\pi_1(\Sigma)\to SL(2,\br)$ be a representation which is elliptic on all boundary components except one of them, $C_1$, where the representation is either elliptic or hyperbolic. Without changing the signature mod $4$, and without changing the representation on $C_1$, one can deform $\phi_0$ to $\phi_1:\pi_1(\Sigma)\to SL(2,\br)$ which has the following property:
\begin{itemize}
  \item If $\phi_0$ is elliptic on $C_1$, then $\phi_1$ is boundary elliptic and contained in a conjugate of $SO(2)$.
  \item If $\phi_0$ is hyperbolic on $C_1$, then $\phi_1$  is elliptic or $-I$ on other boundary components and stabilizes a geodesic in hyperbolic plane.
\end{itemize}
\end{prop}

\begin{proof}
Let $\Sigma$ be an oriented surface of genus $0$ with $n$ boundary components. We pick representatives $C_i\in\pi_1(\Sigma)$ of the boundary components and view $\pi_1(\Sigma)$ as freely generated by $C_1,\ldots,C_{n-1}$. 

Let $\phi_0:\pi_1(\Sigma)\to SL(2,\br)$ be a representation which is elliptic on $C_2,\ldots,C_n$. Let $\phi_0(C_i)$, $i=2,\ldots,n$ (eventually $\phi_0(C_1)$ as well) be conjugate to 
$$\begin{pmatrix}
\cos(\theta_i) &  -\sin(\theta_i)  \\
\sin(\theta_i)      & \cos(\theta_i) 
\end{pmatrix}.
$$
By slightly perturbing $\phi_0(C_2)$, one can assume that 
\begin{equation}
\label{sumtheta}
   \sum_{i=1}^{n-1}\theta_i\not=0\mod 2\pi. \end{equation}
Those $\phi_0(C_i)$ which are elliptic act on hyperbolic plane by rotations with centers $\Omega_i$. If $\phi_0(C_1)$ is hyperbolic, we pick a point $\Omega_1$ on its axis. Let $h_{i,t}$ denote the $1$-parameter group of translations along the geodesic $\Omega_i\Omega_1$ such that $h_{i,1}(\Omega_i)=\Omega_1$. We deform $\phi_0$ into a family $\phi_t$ by conjugating $\phi_t(C_i)$, $i=2,\ldots,n-1$ by $h_{i,t}$. In the course of the deformation, for at most finitely many times $0<t_1<\cdots<t_k<1$, the trace 
$$
\mathrm{trace}(\phi_{t}(C_n))=\mathrm{trace}(\prod_{i=1}^{n-1}\phi_t(C_i))^{-1}
$$
crosses $2$. These are the only times when the signature of $\phi_t$ can jump. 

As $t$ crosses $t_j$, the Toledo invariant is continous, so are the rho invariants of $\phi_t(C_i)$, $i=1,\ldots,n-1$, only $\rro(\phi_t(C_n))$ can be discontinuous, and this happens only if $\mathrm{trace}(\phi_t(C_n))-2$ changes sign. If $\mathrm{trace}(\phi_t(C_n))$ increases, $\phi_t(C_n)$ passes from an elliptic element close to $I$ to a hyperbolic element, $\rro(\phi_t(C_n))$ varies from near $\pm 2$ to $0$, so  or conversely. Therefore $\rro(\phi_t(C_n))$ incurs a jump of $\pm 2$ mod $4$.

\medskip

If $\phi_0(C_1)$ is elliptic, $\mathrm{trace}(\phi_t(C_n))<2$ at $t=0$ and at $t=1$. Indeed, at $t=1$, all $\phi_1(C_i)$, $i=1,\ldots,n-1$, are elliptics fixing $\Omega_1$, so does their product $\phi_1(C_n)^{-1}$ which, by construction (Equation \ref{sumtheta}), is not equal to $I$. Thus $\mathrm{sign}(\phi_t)$ incurs an even number of $\pm 2$ jumps mod $4$, so $\mathrm{sign}(\phi_0)=\mathrm{sign}(\phi_1)$ mod $4$ in this case. Since all $\phi_1(C_i)$ are elliptic and fix $\Omega_1$, they belong to the same conjugate of $SO(2)$. 

\medskip

If $\phi_0(C_1)$ is hyperbolic, all $\phi_1(C_i)$, $i=2,\ldots,n-1$, are elliptic and fix $\Omega_1$. 

Assume first that $n$ is odd. We continue deforming by letting each $\theta_i$, $i=2,\ldots,n-1$, move to $\frac{\pi}{2}$ or $\frac{3\pi}{2}$ (depending wether $\theta_i<\pi$ or $\theta_i>\pi$). The resulting $\phi_2(C_i)$ act on hyperbolic plane by the central inversion through $\Omega_1$. Since there is an odd number of them, their product is again a central inversion through $\Omega_1$. The composition with $\phi_2(C_1)=\phi_0(C_1)$ stabilizes the axis $\delta$ of $\phi_2(C_1)$ and reverses its orientation, hence has a fixed point $\Omega'_n$ on $\delta$. It follows that $\phi_2(C_n)$ acts on hyperbolic plane as the central inversion through $\Omega'_n$. In particular, it is elliptic. Therefore, the number of times $\mathrm{trace}(\phi_t(C_n))$ crosses $2$ is even, and $\mathrm{sign}(\phi_0)=\mathrm{sign}(\phi_2)$ mod $4$ again. By construction, all $\phi_2(C_i)$ stabilize $\delta$, hence the whole representation stabilizes $\delta$.

Finally, assume that $n$ is even. Then we move $\theta_2$ continuously to $\pi$ and each $\theta_i$, $i=3,\ldots,n-1$, is moved to $\frac{\pi}{2}$ or $\frac{3\pi}{2}$. The resulting $\phi_2(C_i)$ act on hyperbolic plane either trivially or by the central inversion through $\Omega_1$, their product acts by a central inversion through $\Omega_1$ as well, and $\phi_2(C_n)$ acts as a central inversion through an other point $\Omega'_n$ of $\delta$. Since it is elliptic, the number of times $\mathrm{trace}(\phi_t(C_n))$ crosses $2$ is even, and $\mathrm{sign}(\phi_0)=\mathrm{sign}(\phi_2)$ mod $4$. Again, the whole representation $\phi_2$ stabilizes $\delta$. 
\end{proof}

\begin{cor}\label{ellipticgenus0}
In genus $0$, the values achieved by the signatures of boundary elliptic representations coincide with the signatures of $SO(2)$-representations, i.e. exactly all integers of the form $2\chi(\Sigma)+4a$, $a=0,\ldots,|\chi(\Sigma)|$. 
\end{cor}

\begin{proof}
According to Proposition \ref{deformation}, every boundary elliptic representation of a genus $0$ surface with boundary has the same signature mod $4$ as a boundary elliptic $SO(2)$-representation. Thus its signature equals $2|\chi(\Sigma)|$ mod $4$. We know from the Milnor-Wood-type inequality \ref{MW} that it lies between $2\chi(\Sigma)$ and $-2\chi(\Sigma)$. Thus it is of the form $2\chi(\Sigma)+4a$, for $a=0,\ldots,|\chi(\Sigma)|$. The converse is provided by Proposition \ref{prop-ell} (2).
\end{proof}

\begin{cor}\label{elliptichyperbolicgenus0}
Let $\Sigma_{0,n}$ be an oriented surface of genus $0$ with $n$ boundary components. The values achieved by the signatures of representations to $SL(2,\br)$ which are hyperbolic on a boundary component and elliptic on all others are exactly all even numbers in the interval $[2\chi(\Sigma),2|\chi(\Sigma)|]$.  
\end{cor}

\begin{proof}
According to Proposition \ref{deformation}, the signature mod $4$ is the same as that of a representation stabilizing a geodesic, and which is generated by central inversions. Looking into the proof, we see that an even number of boundary holonomies are central inversions, and 1 or 2 (depending on the parity of $n$) are $-I$ or hyperbolic. The rho invariant of a lift to $SL(2,\br)$ of a central inversion is $1$ or $-1$. The rho invariant of $-I$ or of a hyperbolic element vanishes. Therefore $\rro(\partial\Sigma_{0,n})$ is the sum of an even number of $\pm 1$, hence an even integer. Since the Toledo invariant vanishes, the signature is even. We know from the Milnor-Wood-type inequality \ref{MW} that it lies between $2\chi(\Sigma_{0,n})$ and $2|\chi(\Sigma_{0,n})|$.

\medskip

Conversely, let us assume first that $n=3$. We pick two distinct points $\Omega$ and $\Omega'$ in hyperbolic plane. We pick two numbers $\theta_1,\theta_2\in\{\frac{\pi}{2},\frac{3\pi}{2}\}$. They specify lifts $g_i$, $i=2,3$, to $SL(2,\br)$ of the central inversions through $\Omega$ and $\Omega'$ respectively, hence a representation of $\pi_1(\Sigma_{0,3})$. Then $g_2g_3$ is hyperbolic, and,
$$
\rro(\partial\Sigma_{0,3})=-2,0,2,
$$
depending on the number of times $\frac{\pi}{2}$ is chosen. Furthermore, since the representation fixes a line in hyperbolic plane, the Toledo invariant vanishes. Therefore, when $n=3$, all even integers between $2\chi(\Sigma_{0,3})=-2$ and $2|\chi(\Sigma_{0,3})|=2$ are signatures of representations which are hyperbolic on one boundary component and elliptic on the others.

Given $n\ge 4$, and an integer $a=1,\ldots,n-2$, for each of the representations of $\pi_1(\Sigma_{0,3})$ encountered above, Proposition \ref{prop-ell} provides us with a boundary elliptic $SO(2)$-representation of $\pi_1(\Sigma_{0,n-1})$, which is equal to $g_2^{-1}$ on a boundary component and has signature $2(n-1)-4a$. Gluing both yields a representation of $\pi_1(\Sigma_{0,n})$ which is hyperbolic on one boundary component and elliptic on the others, and has signature $2n-2-4a-2$, $2n-2-4a$ or $2n-2-4a+2$. In this way, we get all even numbers in the interval $[2\chi(\Sigma),2|\chi(\Sigma)|]$.

\end{proof}

\begin{rem}
Applying Corollary \ref{elliptichyperbolicgenus0} to $\Sigma_{0,3}$ and to $\Sigma_{0,n-1}$ and gluing representations of $\pi_1(\Sigma_{0,3})$ and $\pi_1(\Sigma_{0,n-1})$ which are both hyperbolic on the boundary component to be identified and elliptic on all others, one may believe that every even integer will be achieved as the signature of a boundary elliptic representation of $\pi_1(\Sigma_{0,n})$, contradicting Corollary \ref{ellipticgenus0}.
It is not so. In Corollary \ref{elliptichyperbolicgenus0}, the absolute value of the trace of the hyperbolic boundary holonomy can be prescribed (by adjusting the distance $d(\Omega,\Omega')$), but not its sign. For instance, when $n=3$, the representation of $\pi_1(\Sigma_{0,3})$ with $g_1$ hyperbolic and $g_2,g_3$ central inversions has $\mathrm{trace}(g_1)<0$ when its signature is $0$ and $\mathrm{trace}(g_1)>0$ when its signature is $\pm 2$. So when two representations of $\pi_1(\Sigma_{0,3})$ are glued along a boundary component with hyperbolic holonomy (all others being elliptic), either both have signature $0$, or both have signature $\pm 2$. Thus the signature of the resulting representation of $\pi_1(\Sigma_{0,4})$ is a multiple of $4$.
\end{rem}

\subsection{Genus 1}

\begin{prop}\label{traces}
Let $\Sigma_{1,1}$ denote the genus 1 surface with one boundary component. Let $\psi:\pi_1(\Sigma_{1,1})\to SL(2,\br)$ be a representation. 
\begin{enumerate}
  \item There is a nearby representation $\phi$ which is elliptic or hyperbolic or $I$ on the boundary.
  \item Furthermore
  $$
\mathrm{sign}(\phi)=\begin{cases}
\pm 2 & \text{ if the boundary holonomy is elliptic},\\
0 & \text{ if the boundary holonomy is hyperbolic or }I.
\end{cases}
$$
\end{enumerate}
\end{prop}

\begin{proof}
Let 
$$
\alpha(\lambda)=\begin{pmatrix}
\lambda & 0  \\
0 &  \lambda^{-1}
\end{pmatrix},\quad \alpha^{\pm}(s)=\pm\begin{pmatrix}
1 & s   \\
0 & 1
\end{pmatrix},\quad \beta=\begin{pmatrix}
 a     &  b  \\
  c    &  d
\end{pmatrix}
$$
be elements of $SL(2,\bc)$. Note that every element of $SL(2,\br)$ is conjugate in $SL(2,\bc)$ either to $\alpha(\lambda)$, $\lambda\in\br^*$, to $\alpha(e^{i\theta})$, $\theta\in (0,\pi)\cup(\pi,2\pi)$ or to $\alpha^{\pm}(s)$, $s\in\br$.

We compute
$$
\mathrm{trace}([\alpha(\lambda),\beta])-2=-bc(\lambda-\lambda^{-1})^2,\quad \mathrm{trace}([\alpha^{\pm}(s),\beta])-2=s^2c^2.
$$
We observe that these traces take the values $\pm 2$ only for nongeneric matrices $\beta$. 
\begin{itemize}
  \item If $\mathrm{trace}([\alpha(\lambda),\beta])=2$ and $[\alpha(\lambda),\beta]\not=I$, then $bc=0$ and $\lambda\in\br^*$. 
  \item If $\mathrm{trace}([\alpha^{\pm}(s),\beta])=2$ and $[\alpha^{\pm}(s),\beta]\not=I$, then $c=0$. 
\end{itemize}
In the same manner, the only way to get a commutator with trace $-2$ is to take a commutator with $\alpha(\lambda)$ for $\lambda\not=\pm 1$, and then $\beta$ satisfies the equation $bc(\lambda-\lambda^{-1})^2=4$. This shows that a slight perturbation of $\beta$ produces an elliptic or hyperbolic commutator, unless the commutator is trivial. 

Let us assume that $\beta$ is such that $bc\not=0$. Then
$$
\mathrm{trace}([\alpha(\lambda),\beta])=2 \iff \lambda=\pm 1.
$$ 
Therefore, if $\lambda\in\br^*$, $\alpha(\lambda)$ can be connected by a path $t\mapsto \alpha_t:=\alpha(\mathrm{sgn}(\lambda)|\lambda|^t)$ to $\mathrm{sgn}(\lambda)I$ without $\mathrm{trace}([\alpha(\lambda),\beta])$ crossing $2$. If $\lambda=e^{i\theta}$, $\theta\in(0,\pi)\cup(\pi,2\pi)$, $\alpha(e^{i\theta})$ can be connected by a path $t\mapsto \tilde \alpha_t:=\alpha(e^{it\theta})$ or $\alpha(e^{it(\theta-2\pi)})$ to $I$ without $\mathrm{trace}([\tilde \alpha_t,\beta])$ crossing $2$. 
Also
$$
\mathrm{trace}([\alpha^{\pm}(s),\beta])=2 \iff s=0.
$$
Therefore $\alpha^{\pm}(s)$ can be connected by a path $t\mapsto \alpha^{\pm}_t:=\alpha^{\pm}(ts)$ to $\pm I$ without $\mathrm{trace}([\alpha^{\pm}_t,\beta])$ crossing $2$. 

Let us view $\pi_1(\Sigma_{1,1})$ as freely generated by symbols $A,B$. For a representation $\psi:\pi_1(\Sigma_{1,1})\to SL(2,\br)$, the boundary holonomy is $([\psi(A),\psi(B)])^{-1}$. The perturbation that kills parabolic commutators translates into a perturbation $\phi$ of $\psi$ which makes the boundary holonomy elliptic or hyperbolic or trivial.

The above paths provide continuous deformations $\phi_t$, $t\in(0,1]$ of $\phi_1:=\phi$, converging to representations with abelian image in $PSL(2,\br)$. Along the deformations, the type of the boundary holonomy and hence the signature does not change. When $t$ tends to $0$, the Toledo invariant tends to $0$ and the rho invariant of the boundary tends to $\pm 2$ in the elliptic case, to $0$ in the hyperbolic case. Therefore 
$$
\mathrm{sign}(\phi)=\begin{cases}
0 & \text{ if the boundary holonomy is hyperbolic}, \\
\pm 2 & \text{ if the boundary holonomy is elliptic}.
\end{cases}
$$
When $\phi$ is trivial on the boundary, its image is abelian, hence its Toledo invariant vanishes. So does the rho invariant of the boundary, so the signature vanishes.
\end{proof}

\begin{rem}\label{genus1hyp}
Every hyperbolic conjugacy class of $SL(2,\br)$ arises as the boundary holonomy of a representation of $\pi_1(\Sigma_{1,1})$.
\end{rem}

\begin{proof}
We choose $\psi_{\pm,\lambda}(A)=\alpha(\lambda)$, $\lambda>0$, 
$$
\psi_{-,\lambda}(B)=\begin{pmatrix}
2 & 1 \\
1 & 1
\end{pmatrix}\quad\text{and}\quad\psi_{+,\lambda}(B)=\begin{pmatrix}
0 & -1 \\
1 & 0
\end{pmatrix}.
$$ 
Then 
$$
\mathrm{trace}([\psi_{\pm,\lambda}(A),\psi_{\pm,\lambda}(B])=2\pm(\lambda-\lambda^{-1})^2.
$$
As $\lambda$ varies in $(1,+\infty)$ (resp. $(1+\sqrt{2},+\infty)$), the trace of the commutator takes every value in $(2,+\infty)$ (resp. in $(-\infty,-2)$). Thus all hyperbolic conjugacy classes are hit.
\end{proof}

\begin{cor}\label{elliptgenus1}
Let $\Sigma$ be an oriented surface of genus $1$ with $n\ge 1$ boundary components.
The values achieved by the signatures of boundary elliptic representations of $\pi_1(\Sigma)$ are exactly the following:
\begin{itemize}
  \item If $n=1$, $\pm2$.
  \item If $n\ge 2$, all even integers in $[2\chi(\Sigma),2|\chi(\Sigma)|]$.
\end{itemize}  
\end{cor}

\begin{proof}
The case when $n=1$ is covered by Proposition \ref{traces}.

If $n\ge 2$, Proposition \ref{even} combined with the Milnor-Wood type inequality \ref{MW} implies that $\mathrm{sign}(\phi)$ is an even integer in $[2\chi(\Sigma),2|\chi(\Sigma)|]$. 

\medskip

Conversely, let us assume that $n\ge 2$. Let $m$ be an even integer in $[-2n+2,2n-2]$. Corollary \ref{elliptichyperbolicgenus0} provides us with a representation $\phi''$ of the fundamental group of the genus $0$ surface $\Sigma_{0,n+1}$ with $n+1$ boundary components which is elliptic on $n$ of them and belongs to a hyperbolic conjugacy class $h$ on the last one, and has signature $m$. Remark \ref{genus1hyp} provides us with a representation $\phi'$ of $\pi_1(\Sigma_{1,1})$ with boundary holonomy $h^{-1}$. Its signature vanishes. We glue them together to get a boundary elliptic representation of $\pi_1(\Sigma)$ of signature $m$. 

Finally, the extreme values $\pm 2\chi(\Sigma)$ are achieved by holonomies of hyperbolic surfaces with cone singularities, as in Section \ref{relative}. Thus every even integer in $[2\chi(\Sigma),2|\chi(\Sigma)|]$ is the signature of a boundary elliptic $SL(2,\br)$-representation of $\pi_1(\Sigma)$.

\end{proof}

\subsection{Proof of Theorem \ref{elliptic}}

It is a combination of Proposition \ref{elliptichigher} and Corollaries \ref{ellipticgenus0} and \ref{elliptgenus1}.

\section{Representations to unitary groups}

In this section, we will discuss some partial results on the possible values of signature for the surface group representations in $U(p,q)$.

\subsection{$U(p,q)$ for genus zero surfaces}

  For  any representation $\phi:\pi_1(\Sigma)\to U(p,q)$, by \cite[Theorem 4]{KPW}, we have the following formula
\begin{equation*}
  \mr{sign}(\phi)=-2\mr{T}(\phi)+\rro_\phi(\p\Sigma)\in \mb{Z}.
\end{equation*}
The signature satisfies the following bound
\begin{equation*}
  |\mr{sign}(\phi)|\leq (p+q)|\chi(\Sigma)|. 
\end{equation*}

Now we assume that $\Sigma$ is the surface of genus zero with $n$ ($\geq 2$) boundaries $c_1,\cdots, c_n$, and consider the following representation
\begin{equation*}
  \phi:\pi_1(\Sigma)\to \underbrace{U(1)\times \cdots\times U(1)}_{p}\times \underbrace{U(1)\times \cdots\times U(1)}_{q}\subset U(p)\times U(q)\subset U(p,q)
\end{equation*}
\begin{equation*}
  \phi(c_i)=\mathrm{diag}(e^{i\theta_{i,1}},\cdots,e^{i\theta_i,p},e^{i\theta_i,p+1},\cdots,e^{i\theta_i,p+q}),\quad 1\leq i\leq n,
\end{equation*}
where all $\theta_{i,j}$ in $[0,2\pi)$.
By \cite[(6.7)]{KPW}, one has 
\begin{equation}\label{eqn4}
  \mathrm{sign}(\phi)=\sum_{i=1}^n\sum_{j=1}^p\left(\mathrm{sgn}(\theta_{i,j})-\frac{\theta_{i,j}}{\pi}\right)-\sum_{i=1}^n\sum_{l=p+1}^{p+q}\left(\mathrm{sgn}(\theta_{i,l})-\frac{\theta_{i,l}}{\pi}\right).
\end{equation}
Note that $\phi$ is a representation if and only if 
\begin{equation*}
  \sum_{i=1}^n \theta_{i,j}=2k\pi\text{ for } 1\leq j\leq p+q.
\end{equation*}

In this case, $\chi(\Sigma)=2-n$, and for any integer $m$ in $[-(p+q)(n-2),(n-2)(p+q)]$, there exist integers $m_1,\cdots,m_{p+q}\in [-(n-2),n-2]$ such that
\begin{equation*}
  m=(m_1+\cdots+m_{p})-(m_{p+1}+\cdots+m_{p+q}).
\end{equation*}
It is straightforward to check that there exist $\{\theta_{1,j},\cdots,\theta_{n,j}\}_{1\leq j\leq p+q}$ such that    
\begin{equation*}
\sum_{i=1}^n  \left(\mathrm{sgn}(\theta_{i,j})-\frac{\theta_{i,j}}{\pi}\right)=m_j.
\end{equation*}
From \eqref{eqn4}, we have
\begin{equation*}
   \mr{sign}(\phi)=\sum_{j=1}^p m_j-\sum_{l=p+1}^{p+q}m_l=m.
\end{equation*}

\begin{prop}\label{g=0}
If $\Sigma$ has genus zero, $\chi(\Sigma)\leq 0$, then every integer in the interval $[-(p+q)|\chi(\Sigma)|,(p+q)|\chi(\Sigma)|]$ is the signature of some representation $\pi_1(\Sigma)\to U(p,q)$, which is in $U(1)^{\times p}\times U(1)^{\times q}\subset U(p)\times U(q)\subset U(p,q)$.
\end{prop}

\subsection{$U(p)$}

Since Proposition \ref{g=0} gives a complete answer in geneus $0$, from now on, we assume that surfaces have genus $g\geq 1$.

\begin{rem}
When $g\geq 1$, even for the group $U(p)$, the Milnor-Wood-type bound
$$
|\mathrm{sign}|\le p|\chi(\Sigma)|
$$
need not be sharp. For example, the surface $\Sigma_{1,1}$ of genus one with one boundary has $\chi(\Sigma_{1,1})=-1$. Since $U(1)$ is abelian, the representations in $U(1)$, i.e. $p=1$, are automatically trivial along the boundary. Furthermore, the Toledo invariant vanishes, therefore the signature is always zero and cannot take values $1$ and $-1$. 
\end{rem}

\begin{prop}\label{U}
Let $\Sigma$ be an oriented surface with $n\ge 1$ boundary components, and genus $\ge 1$. Let $p\ge 1$ be an integer. For every representation $\pi_1(\Sigma)\to U(p)$,
$$
|\mathrm{sign}|\le \max\{0,np-2\}.
$$
Conversely, when $np\ge 2$, every integer in the interval $[2-np,np-2]$ is the signature of some representation $\pi_1(\Sigma)\to U(p)$.
\end{prop}

\begin{proof}

By \cite[Section 6.2]{KPW}, for any representation $\phi:\pi_1(\Sigma)\to U(p)$, $\mathrm{T}(\Sigma,\phi)=0$ and 
\begin{align*}
  \mathrm{sign}(\phi)&=\rro_\phi(\p\Sigma)=\sum_{i=1}^n\sum_{j=1}^p\left(\mathrm{sgn}(\theta_{i,j})-\frac{\theta_{i,j}}{\pi}\right)\\
  &\leq np - \frac{1}{\pi}\sum_{i=1}^n\sum_{j=1}^p\theta_{i,j}.
\end{align*}
where $\phi(c_i)\sim \mathrm{diag}(e^{i\theta_{i,1}},\cdots,e^{i\theta_{i,p}})$, $\theta_{i,j}\in [0,2\pi)$. 

Since $U(1)$ is abelian, the representation $\mathrm{det}\circ\phi:\pi_1(\Sigma)\to U(1)$ satisfies 
$$
\prod_{i=1}^{n}\mathrm{det}\circ\phi(c_i)=I,
$$
hence $\sum_{i=1}^n\sum_{j=1}^p\theta_{i,j}=2\pi k$ is an integer multiple of $2\pi$. Therefore $\mathrm{sign}(\phi)\le np-2k$. 

If $k\le 0$, then all $\theta_{i,j}$ vanish, i.e. the representation is trivial on the boundary. This means that $\phi$ arises from a unitary representation of a closed surface group. Then, the signature vanishes.

Otherwise, $k\ge 1$, hence $\mathrm{sign}(\phi)\le np-2$. Since complex conjugation changes the sign of signature, either $\mathrm{sign}(\phi)=0$ or $|\mathrm{sign}(\phi)|\le np-2$.

\medskip

Conversely, remember that $g\geq 1$ and assume that $np\ge 2$.  Given an integer $m\in [2-np,np-2]$, we can pick real numbers $\theta_{i,j}\in [0,2\pi)$ for $1\leq i\leq n$ and $1\leq j\leq p$, such that 
\begin{equation*}
\sum_{i=1}^n\sum_{j=1}^p\left(\mathrm{sgn}(\theta_{i,j})-\frac{\theta_{i,j}}{\pi}\right)=m \quad \text{and} \quad \sum_{i=1}^n\sum_{j=1}^p\theta_{i,j}\in 2\pi\mb{Z}.
\end{equation*}
We define a representation $\phi:\pi_1(\Sigma)\to U(p)$ as follows: we first prescribe
\begin{equation}\label{defn-rep-1}
  \phi(c_i)=\mathrm{diag}(e^{i\theta_{i,1}},\cdots,e^{i\theta_{i,p}}), 1\leq i\leq n.
\end{equation}
By construction, letting $\Theta_j=\sum_{i=1}^{n}\theta_{i,j}$,
\begin{equation*}
  \phi(c_1)\cdots\phi(c_n)=\mathrm{diag}(e^{-i\Theta_1},\cdots,e^{-i\Theta_p})\in SU(p).
\end{equation*}
We denote by $Q:=\mathrm{diag}(e^{i\Theta_1},e^{i\sum_{j=1}^2\Theta_j},\cdots,e^{i\sum_{j=1}^{p-1}\Theta_j})$, and set 
\begin{equation}
A=\left[\begin{array}{c:c}
0 & I_{p-1} \\
\hdashline 1 & 0
\end{array}\right],\quad B=\left[\begin{array}{c:c}
0 & 1 \\
\hdashline Q & 0
\end{array}\right].
\end{equation}
Then $A,B\in U(p)$ and 
\begin{equation*}
  [A,B]:=ABA^{-1}B^{-1}=\mathrm{diag}(e^{i\Theta_1},\cdots,e^{i\Theta_{p-1}},e^{-i\sum_{j=1}^{p-1}\Theta_j})=(  \phi(c_1)\cdots\phi(c_n))^{-1}.
\end{equation*}
Now we define
\begin{equation}\label{defn-rep-2}
  \phi(a_1)=A,\,\phi(b_1)=B,\,\phi(a_i)=\phi(b_i)=I\, \text{ for } 2\leq i\leq g.
\end{equation}
 Then $\phi:\pi_1(\Sigma)\to U(p)$ defined by \eqref{defn-rep-1} and \eqref{defn-rep-2} is a representation, and 
 \begin{equation*}
  \mathrm{sign}(\phi)=\sum_{i=1}^n\sum_{j=1}^p\left(\mathrm{sgn}(\theta_{i,j})-\frac{\theta_{i,j}}{\pi}\right)=m.
\end{equation*}

\end{proof}

Using Propositions \ref{g=0} and \ref{U}, we get that the set of values of the signature for representations $\phi:\pi_1(\Sigma)\to U(p)$ ($p\geq 1$) is as follows:
\begin{equation*}
   \begin{cases}
 [-p(n-2),p(n-2)]	& \text{ for }g=0 \text{ and }n\geq 2\\
 [2-np,np-2]\cup \{0\}	&\text{ for } g\geq 1 \text{ and } n\geq 1.
 \end{cases}
\end{equation*}
This completes the proof of Theorem \ref{U(p)}.

\subsection{$U(p,p)\times U(q-p)$}

 Any representation $\phi:\pi_1(\Sigma)\to U(p,p)\times U(q-p)\subset U(p,q)$ takes the form $\phi=(\phi_1,\phi_2)$, where $\phi_1:\pi_1(\Sigma)\to U(p,p)$ and $\phi_2:\pi_1(\Sigma)\to U(q-p)$ are representations. Using Propositions \ref{g=0} and \ref{U}, the set of values of the signature is as follows:
\begin{equation*}
   \begin{cases}
 [-(p+q)(n-2),(p+q)(n-2)]	& \text{ for }g=0 \text{ and }n\geq 2\\
[-2p(2g-2+n),2p(2g-2+n)] &\text{ for }g\geq 1, n(q-p)\leq 1\\
 [-p(n+4g-4)-qn+2,p(n+4g-4)+qn-2]	&\text{ for } g\geq 1,  n(q-p)\geq 2.
 \end{cases}
\end{equation*}

\begin{rem}
Let us call a representation $\phi \in \operatorname{Hom}\left(\pi_1(\Sigma), U(p, q)\right)$ a \emph{maximal representation of signature} if $\operatorname{sign}(\phi) \geq \operatorname{sign}(\psi)$ for every $\psi \in \operatorname{Hom}\left(\pi_1(\Sigma), U(p, q)\right)$. Analogously, the concept of maximal representation applies to the Toledo invariant, as described in \cite[Definition 2]{BIW}. According to \cite[Theorem 5]{BIW}, maximal representations of the Toledo invariant are contained in conjugates of the subgroup $U(p, p) \times U(q-p) \subset U(p, q)$. This raises the question of whether the same result holds for \emph{maximal representations of signature}. 
\end{rem}

\end{document}